\documentclass[11pt,a4paper, final]{amsart}
\usepackage[utf8]{inputenc}
\usepackage[T1]{fontenc}
\usepackage[margin=3.5cm]{geometry}
\usepackage{amsmath}
\usepackage{amssymb}
\usepackage{graphicx}

\usepackage{mathtools,bbold}
\usepackage{tikz-cd}
\usepackage{comment, stackrel}
\usepackage{hyperref}
\usepackage{mathrsfs}
\usepackage{xspace}

\RequirePackage[notref, notcite]{showkeys}               

\RequirePackage{xspace}         
\RequirePackage{textcomp}       
\RequirePackage{multirow}       
\RequirePackage{tablefootnote}  
\RequirePackage[backend = biber, style = numeric, ibidtracker=true, maxbibnames=4]{biblatex}
\RequirePackage{varioref}
\RequirePackage{hyperref}
\RequirePackage[nameinlink, capitalize, noabbrev]{cleveref}
\RequirePackage{calligra}	

\definecolor{LightGray}{rgb}{0.9,0.9,0.9}
\RequirePackage[color       = LightGray,
bordercolor = LightGray,
textsize    = footnotesize,
figwidth    = 0.99\linewidth,
obeyFinal
]{todonotes} 

\DeclareMathOperator{\Hom}{Hom}					
\DeclareMathOperator{\Ext}{Ext}					
\DeclareMathOperator{\Tor}{Tor}					
\DeclareMathOperator{\id}{id}								
\DeclareMathOperator{\im}{im}					

\DeclareMathOperator{\coker}{coker}				
\newcommand{\ol}[1]{\overline{#1}}							%
\newcommand{\dual}{{}^\vee}									%
\newcommand{\ie}{\leavevmode\unskip, i.e.,\xspace}
\newcommand{\eg}{\leavevmode\unskip, e.g.,\xspace}
\newcommand{\shf}[1]{\mathscr{#1}}							
\newcommand{\OO}{\mathcal{O}}								

\newcommand{\cat}[1]{{\normalfont\mathsf{#1}}}	

\newcommand{\Coh}[1]{\cat{Coh}({#1})}			
\newcommand{\ECoh}[2]{\cat{Coh}_{#1}({#2})}		

\theoremstyle{plain}
\newtheorem{theorem}{Theorem}

\newtheorem{proposition}[theorem]{Proposition}
\newtheorem{lemma}[theorem]{Lemma}
\newtheorem{corollary}[theorem]{Corollary}

\theoremstyle{definition}
\newtheorem{definition}[theorem]{Definition}

\newtheorem*{notation}{Notation}

\theoremstyle{remark}
\newtheorem{remark}[theorem]{Remark}
\newtheorem{remarks}[theorem]{Remarks}

\newcommand{\injto}{\xhookrightarrow{}}						
\newcommand{\SC}{\mathbb{C}}                    
\newcommand{\SL}{\mathbf{SL}}						

\newcommand{\Z}{\mathbb{Z}}    						
\newcommand{\Q}{\mathbb{Q}}    						
\newcommand{\C}{\mathbb{C}}    						
\renewcommand{\P}{\mathbb{P}}  						

\newcommand{\head}{\operatorname{h}}

\newcommand{\tail}{\operatorname{t}}

\newcommand{\CQ}{\mathbb{C}Q}

\newcommand{\QuotI}[1][ ]{\Quot_I^{#1}}
\newcommand{\QuotInI}{\QuotI[n_I]}
\newcommand{\Pibullet}{\Pi^{\bullet}}

\numberwithin{equation}{section}
\numberwithin{theorem}{section}

\newcommand{\Cohfr}[1]{\cat{Coh}^{\mathrm{fr}}{#1}}

\DeclareMathOperator{\sheafHom}					
{
	\mathscr{H}\text{\kern -5.2pt {\calligra\large om}}\,
}
\DeclareMathOperator{\sheafExt}					
{
	\mathscr{E}\text{\kern -3.7pt {\calligra\large xt}}\,
}

\DeclareMathOperator{\Sym}{Sym}					
\newcommand{\Hilb}{\operatorname{Hilb}}						
\newcommand{\Quot}{\operatorname{Quot}}						
\newcommand{\isoto}{\xrightarrow{\sim}}						
\DeclareMathOperator{\Spec}{Spec}				
\DeclareMathOperator{\Proj}{Proj}				

\mathchardef\mhyphen="2D

\title{Noncommutative projective partial resolutions and quiver varieties}
\author{Søren Gammelgaard}
\address{Dipartimento di Matematica e Informatica, Università degli Studi di Ferrara, Via Nicolò Machiavelli, 30, 44121 Ferrara, Italy}
\email{gmmsrn@unife.it}

\author{\'Ad\'am Gyenge}
\address{Budapest University of Technology and Economics, Department of Algebra and Geometry, Institute of Mathematics, M\H{u}egyetem rakpart 3, H-1111, Budapest, Hungary}
\email{gyenge.adam@ttk.bme.hu}

\subjclass[2020]{Primary 14A22; Secondary 16G20, 14E16}
\keywords{noncommutative geometry, quiver variety, McKay correspondence, nongeneric}

\begin{document}

	\begin{abstract}
		
		Let $\Gamma\in \SL_2(\C)$ be a finite subgroup.  We introduce a class of projective noncommutative surfaces $\P^2_I$, indexed by a set of irreducible $\Gamma$-representations.  Extending the action of $\Gamma $ from $\C^2$ to $\P^2$, we show that these surfaces generalise both $[\P^2/\Gamma]$ and $\P^2/\Gamma$. We prove that isomorphism classes of framed torsion-free sheaves on any $\P^2_I$ carry a canonical bijection to the closed points of appropriate Nakajima quiver varieties. In particular, we provide geometric interpretations for a class of Nakajima quiver varieties using noncommutative geometry.
		Our results partially generalise several previous results on such quiver varieties. 
	\end{abstract}
	
	\maketitle
	
	\tableofcontents

	\section{Introduction}

	Given a subgroup $\Gamma\subset \SL_n(\C)$, one can construct the quotient singularity $X=\C^n/\Gamma\coloneqq \Spec (\C[x_1,\dots, x_n])^\Gamma$, and investigate various  spaces attached to $X$: for instance, the Hilbert schemes of points on $X$, the equivariant Hilbert scheme $n\Gamma\textrm{-}\Hilb(V)$, and resolutions of $X$. One can often construct such spaces as quiver varieties. 
	Especially, when $n=2$, such moduli spaces can often be constructed as \emph{Nakajima quiver varieties} (see \cref{subsec:interpretations} for a list). These quiver varieties, first defined in \cite{Nakajima94} have become very useful in representation theory and algebraic geometry, not only because they provide constructions of interesting moduli spaces, but also because they satisfy desirable geometric properties:  they are irreducible, have symplectic singularities, and when smooth, they are hyperk\"ahler.
	
	In this paper, we will find geometric interpretations of a large class of Nakajima quiver varieties. The quiver varieties appearing in this paper will always be built from a particular quiver (the \emph{McKay quiver}) canonically associated to $\Gamma\subset \SL_2(\C)$ by the McKay correspondence \cite{McKay80}. We will use the McKay quiver to build several \emph{noncommutative} `partial resolutions' $\P^2_I$ of the singular scheme $\P^2/\Gamma$, and we show that moduli spaces of \emph{framed sheaves} on $\P^2_I$ have a canonical bijection of closed points to Nakajima quiver varieties.

	In more detail, choose a finite subgroup $\Gamma\subset \SL_2(\C)$. 
	Then the quotient $\C^2/\Gamma$ is a Kleinian singularity (also known as, among other names, a du Val singularity, or a canonical surface singularity). The geometry of such singularities is of importance in the minimal model program and for theoretical physics (see e.g. \cite{Matsuki}, \cite{HeckmanMorrisonVafa}). It has long been recognised that an approach to studying them is through the McKay correspondence \cite{McKay80}. Among other things, this correspondence associates a graph (the \emph{McKay graph}) to the isomorphism class of $\Gamma$, which is an affine Dynkin graph. Replacing every edge in this graph by a pair of opposing arrows, we obtain a quiver, the \emph{McKay quiver} associated to $\Gamma$. From this quiver, with some additional data, one can construct several Nakajima quiver varieties, which often turn out to be isomorphic to - or, at least, be in canonical bijection with - interesting moduli spaces attached to the singularity $\C^2/\Gamma$. In various papers \eg \cite{VaVa, BelSche, BelCra, CGGS, CGGS2, Paper1, BGK}, such moduli spaces were identified with Nakajima quiver varieties; we mention two of them here. We also give an overview of all known (at least, to us) interpretations of Nakajima quiver varieties built from the McKay quiver $Q$ in \cref{subsec:interpretations}.
	
	Let $Q_0 =\{\rho_0, \dots, \rho_r\}$ be the set of irreducible representations of $\Gamma$, such that $\rho_0$ is trivial. For ease of notation, we will identify $Q_0$ with $\{0,\dots,r\}$. 
	
	First, in \cite{CGGS2}, we showed with our coauthors that for any non-empty $I\subseteq Q_0$  and for any dimension vector $n_I \in \mathbb{N}^{|I|}$ there is an \emph{orbifold Quot scheme} $\QuotInI([\C^2/\Gamma])$ parametrising isomorphism classes of quotients of an equivariant rank 1 sheaf on $\C^2$. 
	It was also shown \cite{CGGS2, CrawYamagishi} that  
	$\QuotInI([\C^2/\Gamma])$ is isomorphic to a particular Nakajima quiver variety.
	
	When $I = \{i\}\subset Q_0$ is a singleton with $\dim \rho_i = 1$, there is also an isomorphism \cite[Proposition 3.4]{CGGS2} (see also \cite[Theorem 1.3]{CrawYamagishi})
	$\QuotInI([\C^2/\Gamma]) \isoto \Quot^{n_i}_i (\C^2/\Gamma)$
	where $\Quot^{n_i}_i (\C^2/\Gamma) $ is the classical Quot scheme parameterising $\C[x,y]^\Gamma$- submodules $M$ of $R_i \coloneqq \Hom(\rho_i, \C[x,y])$ such that $\dim R_i/M = n_i$. The Quot scheme corresponding to the subset $I=\{0\}$ is thus isomorphic to the Hilbert scheme of points  $\Hilb^n (\C^2/\Gamma)$. 

	Second, the first author showed in \cite{Paper1} that there is a projective Deligne-Mumford stack $\mathcal{X}$ containing $\C^2/\Gamma$ as an open subscheme. On this stack, one can construct (\cite{Paper2}) moduli spaces of \emph{framed sheaves} of rank $\ge 1$, satisfying certain cohomological conditions. The set of closed points of such a moduli space carries a canonical bijection to the closed points of a corresponding Nakajima quiver variety, again constructed from the McKay quiver of $\Gamma$. This construction turns out to also generalise $\Hilb^n(\C^2/\Gamma)$, which corresponds to the case where the rank is 1. The construction from \cite{Paper1}, however, only focuses on the subset $\{0\}\subset Q_0$. 
	
	In this work, we 
	unify and generalise the results of the above papers. 
	We investigate the higher rank framed sheaves for general index sets $I\subseteq Q_0$ using a noncommutative geometric approach pioneered in \cite{BGK}, and find a description of the moduli spaces of these objects as Nakajima quiver varieties.

	We present two main results. 
	\subsection{Noncommutative projective partial resolutions}
	First, we extend the cornering method from \cite{CIK18} to graded algebras and modules to construct a family of noncommutative projective spaces which generalise several previous constructions and members of which have decent properties. Recall that our index set $Q_0=\{0,\dots,r\}$ corresponds to the irreducible representations of a fixed finite subgroup $\Gamma \subset \SL_2(\C)$.
	\begin{theorem}\label{thm:first}
		To each subset $I \subseteq Q_0$ there corresponds a `noncommutative projective surface' $\mathbb{P}^2_{I}$.

		In particular, our noncommutative $\mathbb{P}^2_{I}$ generalise the following spaces:
		\begin{itemize}
			\item for $I=\{0,\dots,r\}=Q_0$ our space $\mathbb{P}^2_{\Gamma} \coloneqq \mathbb{P}^2_{Q_0}$ is equivalent to the quotient stack $[\mathbb{P}^2/{\Gamma}]$;
			\item for $I=\{0\}$, $\mathbb{P}^2_{0}$ is equivalent to the scheme-theoretic quotient $\mathbb{P}^2/{\Gamma}$,
		\end{itemize} where \emph{equivalent} means that there is a natural equivalence of the respective categories of coherent sheaves. 
	\end{theorem}
	
	\begin{corollary}\label{cor:firstFunctorial}
		If $I\subset J$, there is a canonical morphism $\P^2_J\to \P^2_I$. If $I=\{0\}$, $J = Q_0$, this morphism corresponds to the coarse moduli space morphism $[\P^2/\Gamma]\to \P^2/\Gamma$.
	\end{corollary}
	
	The theorem will follow from \cref{sec:CorneredP2Geometry}, with the spaces $\P^2_I$ being defined in \cref{def:ProjPartRes}. The corollary will be shown in \cref{cor:factorise}. The two special cases mentioned above follow from categorical equivalences that are known to the experts, but the spaces $\P^2_I$ are, as far as we know, new. If $0\in I$, we call $\P^2_I$ a noncommutative \emph{projective partial resolution} of $\P^2/\Gamma$.

	Moreover, there is a noncommutative divisor $\P^1_I\subset \P^2_I$, the complement of which coincides with the noncommutative \emph{affine} scheme whose rank 1 Quot scheme was described in \cite{CGGS2}. (We will make this assertion precise in \cref{cor:IsEquivQuotScheme}.)
	The spaces $\P^2_I$ can, if $0\in I$, be interpreted as partial noncommutative resolutions of projectivisations of Kleinian singularities. They carry rich geometric structures, for instance a collection of 'twisting' line bundles, and hence they provide an interesting family of examples for noncommutative algebraic geometry.

	\subsection{Interpretations of Nakajima quiver varieties}
	Our second main result is that we enhance the main results of \cite{VaVa}, \cite{CGGS2}, and \cite{Paper1} (on the level of bijections of points) by giving a quiver description of moduli spaces of torsion free sheaves on $\mathbb{P}^2_{I}$. We do this using Nakajima quiver varieties of a specific type that we denote as $\mathfrak{M}_{\theta_I}(V, W)$ (we do not give the definition here, see \cref{def:NQV} for it). We will so far only note the Nakajima quiver varieties appearing here depend on two finite-dimensional $\Gamma$-representations $V, W$, and a subset $I\subset Q_0$ of irreducible representations. Let $V = \bigoplus \rho_i^{\mathbf v_i}$ be a $\Gamma$-representation satisfying a mild constraint on  the vector $\mathbf v = (\mathbf v_i)_{i\in Q_0}$, and let $W= \bigoplus \rho_i^{\mathbf w_i}$ be a $\Gamma$-representation with $\mathbf w_i=0$ for $i\not\in I$. Let $V_I = V|_I$, and denote by $\Cohfr{\P^2_I}(V_I, W)$ the set of isomorphism classes of pairs  $(\shf F, \phi_{\shf F})$ of a torsion free sheaf on $\P^2_I$ and a \emph{$W$-framing} at $\P^1_I$ such that $H^1(\mathbb{P}^2_{I}, {\shf F}(-1))=V_I$ (for the details, see \cref{subsec:framing}).

	Then we prove the following result in \cref{thm:main}.
	\begin{theorem}\label{thm:NQVbijectionIntro}
		There is a Nakajima quiver variety $\mathfrak{M}_{\theta_I}(V, W)$ and a canonical bijection
		
		\[\Cohfr{\P^2_I}(V_I, W)\isoto \mathfrak{M}_{\theta_I}(V, W)(\C).\]
	\end{theorem}

	Our main tool to show these two main theorems will be to introduce a handful of functors: We shall need a functor $\tau_*\colon \Coh{\P^2_{\Gamma}}\to \Coh{\P^2_I}$ and its left adjoint $\tau^*\colon  \Coh{\P^2_I}\to \Coh{\P^2_{\Gamma}}$.  As hinted at by the notation, these behave like the usual direct and inverse image functors from commutative algebraic geometry. 
	Although the definition of the noncommutative schemes $\mathbb{P}^2_{I}$ is straightforward, the existence of our direct image and pullback functors and the fact that they behave as expected is highly nontrivial. We also introduce a `restriction to the divisor at infinity' $c^*\colon \Coh{\P^2_I}\to \Coh{\P^1_I}$. All these functors are defined on the level of quiver representations. Especially $\tau^*$ and $\tau_*$ are induced by two functors arising from a \emph{recollement} of quiver varieties \cite{CIK18}.
	
	As already mentioned, we also show that \cref{thm:NQVbijectionIntro} generalises several earlier results on interpretations of Nakajima quiver varieties.
	In order to do this, we will also have to introduce an `open restriction' functor $r^*: \Coh{\P^2_I}\to \Pi_I\cat{-mod}$ where $\Pi_I\cat{-mod}$ should be thought of as the category of coherent sheaves on the \emph{affine} noncommutative scheme $\P^2_I\setminus \P^1_I$.

	A natural question is whether the sets $\Cohfr{\P^2_I}(V_I, W)$ can be made to carry an intrinsic scheme structure, and then whether the bijection of \cref{thm:NQVbijectionIntro} can be extended to an isomorphism of schemes. We are, however, not aware of any way of constructing moduli spaces of sheaves on noncommutative varieties as schemes or stacks, which is why we will consistently focus on establishing canonical bijections. A naive solution is to define the scheme structure on the moduli side by transfering the one from the quiver side. 
	
	One importance of the varieties  $\mathfrak{M}_{\theta_I}(V, W)$ and hence, of our spaces $\P^2_I$, is that in many cases we can build up modular forms from their topological invariants \cite{BertschGyengeSzendroi}. In these computations, a special role is played by the variation of GIT map which is closely related to $\tau_{\ast}$ and $\tau^{\ast}$ presented here. Therefore, a detailed understanding of their geometry and construction is necessary.
	Again, a scheme structure on $\Cohfr{\P^2_I}(V_I, W)$ would enable us to compare the topology of the two sides directly.

	\subsection{Quiver varieties as moduli spaces}
	\label{subsec:interpretations}
	We give here a brief overview of the already known interpretations of Nakajima quiver varieties built from McKay quivers associated to finite subgroups $\Gamma\subset \SL_2(\C)$. For ease of notation, it is customary to 
	write $\mathbf{v}, \mathbf{w} \in \mathbb{N}^{r+1}$ for the vectors collecting the multiplicities of the irreducible representations in $V$ and $W$ respectively. Then the quiver variety $\mathfrak{M}_{\theta}(V, W)$ is also denoted as $\mathfrak{M}_{\theta}(\mathbf{v}, \mathbf w)$, where $\theta$ is the stability parameter.
	
	Let us first note that in \cite[Section 4.5]{BellamyCrawSchedler} (cf.\cite{BelCra}) a complete description of the wall-and-chamber structure is given for the space of stability parameters, for any pair of vectors $\mathbf{w}, \mathbf v$.
	Let $\delta, \ol 1 \mathbb{N}^{r+1}$ be the dimension vectors given respectively by $\delta_i = (\dim \rho_i)$, $\overline{1}\coloneqq (1,0,\dots,0)$, i.e. $(\ol 1)_i=1$ only for $i =0$, corresponding to the trivial $\Gamma$-representation.  We then have:
	
	\begin{itemize}
		\item Set $\theta = 0$, $\mathbf v = n\delta$, and $\mathbf{w} =\overline{1}$ or $\mathbf{w} = 0$. Then, by \cite[Lemma 4.5]{BelCra} (cf. \cite[Proposition 33]{Kuznetsov07}) , there is an isomorphism $\mathfrak{M}_{0}(n\delta, \ol 1)\isoto \Sym^n(\C^2/\Gamma)$.
		\item There is a chamber $C^+$, such that for $\theta\in C^+$, $\mathfrak{M}_{\theta}(\mathbf{v}, \mathbf w)(\C)$ carries a canonical bijection to the closed points of a moduli space of $\Gamma$-equivariant framed sheaves on $\P^2$ (\cite[Theorem 1]{VaVa}).
		\item There is a chamber $C^-$, such that for $\theta\in C^-$, $\mathfrak{M}_{\theta}(\mathbf{v}, \mathbf w)$ is isomorphic to a moduli space of framed sheaves on a projective Deligne-Mumford stack containing the minimal resolution $\widetilde{\C^2/\Gamma}$ of $\C^2/\Gamma$ as an open substack (\cite{NakALE}). 
		\item For the special case $\mathbf w = \ol 1$ of the above, if $\theta\in C^-$, then there is an \emph{isomorphism} $\mathfrak{M}_{\theta}(n\delta, \ol 1) \isoto \Hilb^n(\widetilde{\C^2/\Gamma})$ \cite{Kuznetsov07}.
		\item For a general $\theta$ lying in $\delta^\perp$, a particular wall of $C^-$, there is an isomorphism $\mathfrak{M}_{\theta}(n\delta, \ol 1)\isoto \Sym^n \widetilde{\C^2/\Gamma}$ (\cite{Kuznetsov07}).
		\item There is a particular ray $\langle\theta_0\rangle$ in the boundary of $C^+$, such that if $\mathbf{w} = (r, 0,\dots, 0)$ for a positive integer $r$, the closed points of $\mathfrak{M}_{\theta_0}(\mathbf{v}, \mathbf w)$ are in canonical bijection with the closed points of a moduli space of framed sheaves on $\mathcal{X}$, a projective Deligne-Mumford stack containing $\C^2/\Gamma$ as an open substack. (\cite{Paper1}).
		\item For the special case $\mathbf w = \ol 1$ of the above, there is an \emph{isomorphism} $\mathfrak{M}_{\theta_0}( n\delta, \ol 1)\isoto \Hilb^n(\C^2/\Gamma)$ (\cite[Corollary 6.3]{CGGS}, \cite[Corollary 6.6]{CrawYamagishi})
		\item For $\theta_I$ lying in any boundary of $C^+$, and for $\mathbf{w} = \ol 1$, it was shown in \cite[Corollary 6.7]{CrawYamagishi} that $\mathfrak{M}_{\theta}(n\delta, \ol 1)$ is isomorphic to an 'orbifold Quot scheme' (see \cite[Section 3]{CGGS2} for the definition).
		
	\end{itemize}
	We remark that the arguments of \cite{CGGS, CGGS2}, purporting to prove the last two points in the overview above on the level of (underlying reduced) schemes contained a flaw, which was recently patched by \cite{CrawYamagishi}.

	\subsection{Conventions and notation}
	We work throughout over $\C$, thus all algebras appearing are $\C$-algebras. Given an algebra $R$, we write $R-\cat{mod}$ for the category of \emph{left} $R$-modules, $\cat{mod}-R$ for the category of \emph{right} $R$-modules. 
	
	We are primarily concerned with left modules. Thus \emph{module} means, unless indicated otherwise, \emph{left} module, and $\Hom_R(-,-)$ denotes the homomorphisms in $R-\cat{mod}$. When defining a functor, we will also only write out a definition for \emph{left} modules. In such cases, we will use the subscript $_{\mathrm{right}}$ for the corresponding functor defined for \emph{right} modules.
	Note that one of our main sources, \cite{BGK}, is concerned with \emph{right} modules. When using results from \emph{ibid}, we will consistently adapt their results to left modules, without remarking on this.
	
	All gradings appearing will be $\Z$-gradings. We will indicate the degree $d$-part of a graded algebra or module $A$ by $A_d$, in the case where $A=A_F$ already has a subscript, by $A_{F, d}$.
	
	Given a scheme (or stack) $X$, a \emph{point} of $X$ always means a $\C$-valued point.
	
	Finally, we shall write, given an arrow $a$ in a quiver $Q$, $\tail(a)$ for the tail of $a$, $\head(a)$ for the head.
	
	\subsection{Structure of the paper}
	\Cref{sec:preliminaries} collects the necessary background materials on noncommutative projective geometry, quiver algebras, and Nakajima quiver varieties. In \cref{sec:QuiverAlgs} we introduce the crucial technique of \emph{cornering}, which we use in \cref{sec:CorneredP2Geometry} to define the `noncommutative partial projective resolutions' mentioned in \cref{thm:first}, proving the same Theorem. The same section also defines the functors $\tau_*, \tau^*$. Then, in \cref{sec:MoveToNQV}, we define what it means for a coherent sheaf on $\P^2_I$ to be \emph{framed}, and we finally prove \cref{thm:NQVbijectionIntro}. Following this, \cref{sec:connections} is devoted to showing that \cref{thm:NQVbijectionIntro} indeed generalises \cite[Theorem 1.2]{Paper1} and (on the level of closed points) \cite[Corollary 6.7]{CrawYamagishi}. Following the main content, we show in \cref{sec:Appendix} that the stack $\mathcal{X}$ can in fact be replaced by $\P^2/\Gamma$ in point 5) in the list of \cref{subsec:interpretations}.
	
	\subsection*{Acknowledgments}
	We thank Bal\'azs Szendr\H oi, Alastair Craw, Ugo Bruzzo, Andrea Ricolfi, Erzsébet Lukács, Michele Graffeo, and Joachim Jelisiejew for helpful comments and questions. We thank Nanna Gammelgaard for pointing out several typos. We thank Okke van Garderen for pointing out a mistake in a previous version. S.G. was supported by a SISSA Mathematical Fellowship. Á.Gy.~was supported by the János Bolyai Research Scholarship of the Hungarian Academy of Sciences and by the  “Élvonal (Frontier)” Grant KKP 144148.
	\section{Preliminaries}\label{sec:preliminaries}
	
	\subsection{Noncommutative geometry}
	
	We collect here some results and definitions from \cite{BGK}.
	
	Let $R = \bigoplus_{i\ge 0}R_i$ be a finitely generated graded algebra.
	Let $\cat{gr}(R)$ be the category of finitely generated graded $A$-modules, and put 
	\[\cat{qgr}(R) \coloneqq \cat{gr}(R)/\cat{tor}(R),\]
	the quotient by the Serre subcategory $\cat{tor}(R)$ of finite-dimensional graded modules. In more detail, the  objects of $\cat{qgr}(R)$ are the same as those of $\cat{gr}(R)$, while
	\[ \mathrm{Hom}_{\cat{qgr}(R)}(M,N) = \varinjlim \mathrm{Hom}_{\cat{gr}(R)}(M',N/N')\]
	where the limit is taken over all submodules $M' \subset M, N'\subset N$ such that $M/M', N' \in \cat{tor}(R)$, the set of such pairs $(M', N')$ being ordered by $(M'_1, N'_1)\ll(M'_2, N'_2)$ if and only if $M'_2\subset M'_1,$ and $ N'_1\subset N'_2$. The construction induces a quotient functor
	\[ \pi: \cat{gr}(R) \to \cat{qgr}(R).\]
	
	If $R$ is commutative, Serre's theorem \cite[II.5.15.]{Har77} shows us that  $\cat{qgr}(R)$ is indeed equivalent to $\Coh{\Proj R}$.
	In noncommutative algebraic geometry the category $\cat{qgr}(R)$ 
	is thought of \cite[page~236]{artin1994noncommutative} as the category of coherent sheaves on a noncommutative scheme $X = \mathrm{Proj}(R)$: 
	\[ \Coh{\Proj R}) \coloneqq  \cat{qgr}(R).\]
	We will thus use the word \emph{sheaf} for an object of $\cat{qgr}(R)$. 
	We  write $\OO_X = \pi(_RR)$, where $_RR$ is the ring $R$ considered as a left module over itself. 
	Given a sheaf $\shf F\in \Coh{X}$  and an integer $d$, choose an $R$-module $M$ to represent $\shf F$. We set $\shf F(d) = \pi(M(d))$, where $M(d)$ is given by $M(d)_i = M_{d+i}$.
	Furthermore, we set $ \Ext^p_{X}(\shf F, \shf G) =\Ext^p(\shf F, \shf G) $ to be the $p$th derived functor of $\Hom_{\cat{qgr}(R)}(\shf F, \shf G)$, and $H^p(X, \shf F) = \Ext^p(\OO_{X}, \shf F)$. 
	
	We can, to a limited extent, make "internal" versions of these functors as follows. We  have, for every $i\in \Z$ and every positive $k$, a map $R_k\to \Hom_X(\OO_{X}(i), \OO_X(i+k))$ given by left multiplication. This induces, for any $\shf F\in \Coh{X}$ a \emph{right} $R$-module structure on $\bigoplus_{k\ge 0}\Hom_X(\shf F, \OO_{X}(k))$, and we can set \[\shf F\dual\coloneqq \sheafHom_X(\shf F, \OO_X) = \pi_{\mathrm{right}}\left(\bigoplus_{k\ge 0}\Hom_X(\shf F, \OO_{X}(k))\right) .\] The functor $\shf F\mapsto \sheafHom_X(\shf F, \OO_X)$ is 
	right exact, and we denote its derived functors by $\sheafExt^p(\shf F, \OO_X)$. One can also show that \[\sheafExt^p(\shf F, \OO_X)= \pi_{\textrm{right}}\left(\bigoplus_{k\ge 0}Ext^p(\shf F, \OO_X(k))\right).\]

	In general, for the category $\Coh{X}$ to behave well, we need $R$ to satisfy some additional constraints, for instance that of \emph{strong regularity}, see \cite[Definition 1.1.1]{BGK}, cf. \cite{artin1994noncommutative} for the definition.
	Especially, if $R$ is strongly regular, there are two attached integers $d, l$ (\emph{dimension} and \emph{Gorenstein parameter}), such that
	\begin{enumerate}
		\item $\Ext^{> d}(\shf F, \shf G) = 0$ for any $\shf F, \shf G\in \Coh{\Proj R}$,
		\item $\Ext^d_{\cat{mod}R}(R_0, R) = R_0(l)_{\mathrm{right}}$, $\Ext^i_{\cat{mod}R}(R_0, R) = 0  $ for all $i\ne d$.
	\end{enumerate}
	
	If $R$ is strongly regular, there is an important implication for $\Coh{\Proj R}$.
	
	\begin{proposition}(\textbf{Serre Duality.})\label{prop:SerreDuality}
		Let $R$ be a strongly regular algebra of dimension $d$ and Gorenstein parameter $l$, and let $\shf F, \shf G\in \Coh{\Proj R}$.  Then we have functorial isomorphisms
		\[\Ext^p(\shf F, \shf G) \isoto \Ext^{p-d}(\shf G, \shf F(-l))\dual\]
	\end{proposition}
	\begin{proof}
		This is \cite[Theorem 2.3]{YekutieliZhang}, c.f. \cite[p.3]{artin1994noncommutative}.
	\end{proof}
	
	\begin{definition}[\cite{BGK}]\label{def:sheafTypes} 
		Assume that $\Coh{X} = \cat{qgr}(R)$ for a strongly regular algebra $R$.
		We say that a sheaf $\shf F\in \Coh{X}$ is:
		\begin{enumerate}
			\item \emph{locally free} if $\sheafExt^p(\shf F, \OO_X) = 0$ for all $p>0$,
			\item \emph{torsion-free} if it embeds into a locally free sheaf,
			\item \emph{Artin} if $H^{>0}(X, \shf F(k))=0$ for all $k\in \Z$.
		\end{enumerate}
	\end{definition}
	
	If $X$ is a smooth commutative projective variety, the definitions above are equivalent to the usual definitions (\cite[3]{BGK}).

	\subsection{Algebras associated to $\Gamma$}
	\label{subsec:mckay}
	Once and for all fix a finite group $\Gamma \subset \SL(2,\SC)$. We list the irreducible representations of $\Gamma$ as $\{\rho_0, \rho_1, \dots, \rho_r\}$, where $\rho_0$ is the trivial representation.
	
	Let $R=\SC[x,y,z]$. The group $\Gamma$ acts on the subring $\C[x,y] \subset R$: for $g\in \Gamma$ and $f\in R$, we have $(g\cdot f)(v) = f(g^{-1} v)$ for all $v\in V$. We extend this action by setting $g\cdot z = z$ for every $g\in \Gamma$. Geometrically, this means that we have extended the $\Gamma$-action on $\C^2$ to one on $\P^2$ such that group action restricts to an action on the additional divisor $\{z=0\}$. 
	
	Let $R^\Gamma$ denote the $\Gamma$-invariant subring of $R$.
	The $R^\Gamma$-module $R$ then decomposes into isotypical components \cite[Proposition~4.1.15.]{goodman2009symmetry}:
	\begin{equation}
		\label{eqn:decompR}
		R\cong \bigoplus_{0\leq i\leq r} R_i\otimes_\mathbb{C} \rho_i,
	\end{equation} 
	where $R_i\coloneqq \Hom_\Gamma(\rho_i,R)$ is an 
	$R^\Gamma$-module; note that $R_0=R^{\Gamma}$.

	Consider the skew-group algebra
	\[ S \coloneqq \SC[x,y,z] \ast \Gamma = R \ast \Gamma. 
	\]
	This algebra has a natural grading, defined by $\mathrm{deg}\, x =\mathrm{deg}\, y = \mathrm{deg}\, z = 1$, and
	$\mathrm{deg}\, \gamma = 0$ for any $\gamma \in \Gamma$.
	
	\begin{lemma}\label{lem:EmbeddingSkewGroup} The embedding of the graded algebras $\SC[x,y,z] \subset S$ induces an equivalence of categories 
		\[\Coh{\mathrm{Proj}(S)} \to \ECoh{\Gamma}{\mathrm{Proj}(\SC[x,y,z])}=\ECoh{\Gamma}{\mathbb{P}^2}\]
		where $\ECoh{\Gamma}{\mathbb{P}^2}$ is the category of $\Gamma$-equivariant coherent sheaves on the commutative projective plane $\mathbb{P}^2$.
	\end{lemma}
	\begin{proof} Follows from the definitions (cf. \cite[Proposition~3.2.7.]{BGK})
	\end{proof}

	\subsection{McKay quivers}
	We now associate two undirected graphs to the isomorphism class of $\Gamma$. First, the \emph{McKay graph} of $\Gamma$ has vertex set \[\{0,1, \dots, r\}\] corresponding to the irreducible representations of $\Gamma$, where for each $0\leq i, j\leq r$, we have \[\dim\Hom_{\Gamma}(\rho_j , \rho_i \otimes V )\] edges joining vertices $i$ and $j$. 
	
	Second, let $W = \bigoplus \rho_i^{\mathbf w_i}$ be a representation of $\Gamma$, and set $\mathbf{w}$ to be the vector $(\mathbf w_0, \dots, \mathbf w_r) \in \Z_{\ge 0}^{Q_0}$. 
	Applying the trick of Crawley-Boevey \cite{CrawBoe} we get the \emph{framed McKay graph} as follows: extend the original McKay graph with a single framing node $\infty$ and connect this new node to node $i$ with $\mathbf{w}_i$ edges, $0 \leq i \leq r$. The resulting graph then has vertex set
	\[ \{0,\dots ,r\} \cup \{\infty\}.\]
	
	To each undirected graph above we will associate a quiver, obtained as follows. Let $A_0$ be the set of vertices of the graph, and  $A_1$ the set of all pairs of an edge and an orientation of the same edge, thus $(A_0, A_1)$ form a quiver. We denote by $Q = (Q_0, Q_1)$, resp. $Q^{\mathbf{w}} = (Q^{\mathbf{w}}_0, Q^{\mathbf{w}}_1) $ the quivers obtained in this way from the McKay graph, resp. the framed McKay graph. 
	These are called the  \emph{doubled McKay quiver} and the \emph{doubled framed McKay quiver} respectively. Given an arrow $a$ in either quiver, write $\ol{a}$ for the opposing arrow \ie $a, \ol a$ correspond to the two possible orientations of the same edge.
	
	We now introduce the \emph{preprojective algebra}; in Section~\ref{subsec:graded_preprojective_algebra} below a variant of this, the graded preprojective algebra will be defined.
	Consider the following two-sided ideal of $\CQ$:
	\begin{equation}
		\label{eqn:preprojectiverelation}
		\mathcal J\coloneqq \left\langle \sum_{\overset{a\in Q_1}{\head(a)=i}}(a\ol a -\ol a a) \mid {i \in Q_0}
		\right\rangle.
	\end{equation}
	Replacing $Q$ by $Q^{\mathbf w}$ in this equation, we can similarly define an ideal $\mathcal{J}^\mathbf{w}\subset \C Q^\mathbf{w}$.
	
	\begin{definition}\label{def:preprojAlgs}
		We define the \emph{preprojective algebra} $\Pi $ by $\Pi \coloneqq \C Q/\mathcal{J}$.
		Similarly, define the \emph{framed preprojective algebra} $\Pi^\mathbf{w}$ by $\Pi^\mathbf{w} \coloneqq \C Q^\mathbf{w}/\mathcal{J}^{\mathbf{w}}$.
	\end{definition}
	
	Each vertex $i\in Q_0$, respectively $Q^{\mathbf w}_0$ gives rise to a vertex idempotent $e_i\in \Pi$ or $\Pi^{\mathbf{w}}$. Given a set of vertices $I$, we shall write $e_I \coloneqq \sum_{i\in I} e_i$.
	
	A module $M$ over either of the algebras just defined consists of a vector space $M_i = e_iM$ for each vertex $i$, and for every arrow $a$ a linear map $M_{\head(a)}\to M_{\tail(a)}$ satisfying the relations $\mathcal J$, respectively $\mathcal J^\mathbf{w}$. If $M$ is finite-dimensional, we can thus speak of its \emph{dimension vector} $\dim M = (\dim M_i) = (\dim e_iM)$, $i$ running over the vertices. If $M$ is a $\Pi^w$-module, we will often write $\dim M$ as $(\dim_\infty M, \mathbf v)$, where $\mathbf v\in \Z_{\ge 0}^{Q_0}$ is the dimension vector for the vertices except $\infty$.
	\subsection{Nakajima quiver varieties associated to $\Gamma$}
	Now we introduce the \emph{Nakajima quiver varieties} associated to $\Gamma$. We shall only use them in \cref{sec:MoveToNQV}.
	
	Recall the definition of the framed preprojective algebra $\Pi^{\mathbf w}$.
	We call a tuple $\theta\in \Q^{Q^\mathbf w_0}$ a \emph{stability parameter}. We say that a $\Pi^{\mathbf w}$-module $M$ is $\theta$-semistable, respectively $\theta$-stable, if
	\begin{enumerate}
		\item $\theta(\dim M) = 0$, and
		\item for every nonzero submodule $N\subset M$, we have $\theta(\dim N)\ge 0$, respectively $>0$.
	\end{enumerate}
	Given a dimension vector $ (1,\mathbf v) $ for $\Pi^{\mathbf w}$-modules, we say that a stability parameter $\theta$ is \emph{generic} if every semistable $\Pi$-module of dimension $ (1,\mathbf v) $ is stable. A $\Pi$-module is \emph{polystable} if it is a direct sum of $\theta$-stable modules.
	
	Then two $\Pi^\mathbf{w}$-modules $M, N$ of the same dimension are $S_\theta$-equivalent if they have submodule series
	\[0\subset M_0\subset\cdots \subset M_{t_1} = M, \quad 0\subset N_0\subset\cdots \subset N_{t_2} \subset N\]
	such that every submodule $Q$ appearing has $\theta(\dim Q)>0, and \bigoplus M_{i}/M_{i-1} \cong \bigoplus N_j/N_{j-1}$.
	
	The following theorem was proved in \cite{King94}, using a GIT quotient construction.
	\begin{theorem}\label{thm:NQVMod}
		There exists a coarse moduli space $\mathfrak M_{\theta}(\mathbf v, {\mathbf w})$ of $\theta$-semistable $\Pi^{\mathbf w}$-modules of dimension vector $(1, \mathbf v)$. Every point in $\mathfrak M_{\theta}(\mathbf v, {\mathbf w})$ corresponds to a $S_\theta$-equivalence class, thus $\mathfrak M_{\theta}(\mathbf v, {\mathbf w})$ is a fine moduli space if $\theta$ is generic.
	\end{theorem}
	\begin{definition}[\cite{Nakajima94}]\label{def:NQV}
		The scheme $\mathfrak M_{\theta}(\mathbf v, {\mathbf w})$ is a \emph{Nakajima quiver variety}.
	\end{definition}
	As already mentioned, we will often, given explicit $\Pi_0$-graded (i.e. $Q_0$-graded) vector spaces $V,W$, simply write $\mathfrak M_{\theta}(V, W)$ for $\mathfrak M_{\theta}(\dim V, \dim W)$.
	
	\begin{remarks}
		The original definition of a Nakajima quiver variety (in \cite{Nakajima94}) was as a hyperk\"ahler quotient, our definition is proved equivalent by Kuznetsov (\cite{Kuznetsov07}), using \cite{King94} and \cite{CrawBoe}.
		
		It is unknown whether Nakajima quiver varieties are reduced in general. Thus one sometimes explicitly works with the underlying reduced subscheme of the GIT quotient, this is done in e.g. \cite{CGGS}. We will in this paper only care about the underlying sets of closed points, so reducedness issues do not matter.
		
	\end{remarks}
	
	\begin{lemma}[{\cite{BelSche}}]
		The Nakajima quiver variety $\mathfrak M_{\theta}(\mathbf v, \mathbf{w})$ is quasiprojective, has symplectic singularities, is integral and irreducible. If $\theta$ is generic (i.e. if every $\theta$-semistable $\Pi^{\mathbf w}$-module of dimension $(1, \mathbf v)$ is stable), then $\mathfrak M_{\theta}(\mathbf v, \mathbf{w})$ is hyperk\"ahler.
	\end{lemma}
	
	For any $\mathbf w, \mathbf v\in \Z_{\ge 0}^{Q_0}$, and a choice $I\subset Q_0$, we define $\theta_I$ to be the stability parameter for $\Pi^{\mathbf w}$-modules of dimension $(1,\mathbf v)$ to be that given by: \[\theta_I(e_i) =\begin{cases}
		1  &\quad\text{if } i\in I\\
		0 &\quad\text{if }i \in Q_0\setminus I\\
		-\sum_{i\in I}\mathbf v_i  &\quad\text{if } i=e_\infty\\
	\end{cases}\]
	We will also from now on write $\theta$ for $\theta_{Q_0}$.
	
	Let $M$ be a $\Pi^\mathbf{w}$-module of dimension $(1,\mathbf v)$. It follows from \cref{thm:NQVMod} that if $I\subset J\subset Q_0$, there is a morphism $\mathfrak M_{\theta_J}(\mathbf{v}, \mathbf w)\to \mathfrak M_{\theta_I}(\mathbf{v}, \mathbf w)$.

	One of the main results of this paper, given in \cref{thm:main} below, is, as mentioned, a description of the quiver varieties $\mathfrak M_{\theta_I}(\mathbf v, \mathbf w)$ for such $\theta_I$, subject to some restrictions on $\mathbf w, \mathbf v$.
	
	\subsection{Graded preprojective algebra}
	\label{subsec:graded_preprojective_algebra}
	
	We construct a \emph{graded} preprojective algebra $\Pi^{\bullet}$ as follows. Let $Q = (Q_0, Q_1)$ be the doubled quiver of the McKay graph of $\Gamma$. Construct a "tripled" quiver $Q^\bullet = (Q^\bullet_0, Q^\bullet_1)$ by adding to the double quiver a loop $d_i$ at each vertex $i$, so $Q^\bullet_0 = Q_0$. 
	Consider the following ideals of $\C Q^\bullet$:
	\begin{enumerate}
		\item\label{it:rel1} $\mathcal I_1 $, generated by the expressions $ d_ia-ad_j$ for all $a \in Q_1$ and all $i\in Q_0$ such that $h(a)=j$, $t(a)=i$,
		\item\label{it:rel2} $\mathcal I_2 = \left\langle \sum_{\overset{a\in Q_1}{\head(a)=i}}(a\ol a -\ol a a) \mid {i \in Q_0}
		\right\rangle$.
	\end{enumerate}
	
	\begin{definition}\label{def:GradedPreproj}
		We set $\Pi^{\bullet} =\C Q^\bullet/(\mathcal I_1+\mathcal I_2)$.
	\end{definition}
	This algebra was introduced in \cite[Appendix A]{BGK}, with a more general construction. One could also define $\Pi^\bullet$ as the Jacobian algebra $\operatorname{Jac}(Q^\bullet, W)$ for the \emph{potential} given by  \[W = \left(\sum_{i\in Q_0}d_i\right)\sum_{a\in Q_1}(a\ol a -\ol a a),\] see, e.g., \cite[Section 4.2]{Ginzburg06} for this viewpoint, but we shall use the definition given above.

	Note that the ideal of relations $\mathcal I_1+\mathcal I_2$ is generated by homogeneous elements. Therefore, the algebra $\Pi^{\bullet}$ inherits a grading by \emph{path length} from $\C Q^\bullet$.
	
	\begin{notation}\label{not:zElement}
		From now on, $\Pibullet$ will always denote this graded preprojective algebra, constructed from the McKay graph of $\Gamma$.
		We also fix some notation here: For a vertex $i\in Q_0$, let $z_i$ be the element of $\Pibullet$ arising from the loop at vertex $i$. We set $z = \sum_{i\in Q_0} z_i$, and for an arbitrary $I\subset Q_0$, we set $z_I = \sum_{i\in I} z_i$.
	\end{notation}
	\section{Quivers and quiver algebras}\label{sec:QuiverAlgs}
	
	\subsection{Morita equivalences}
	The reason to introduce the algebra $\Pibullet$ is that the category $\Coh{\Proj \Pibullet}$ is equivalent to $\ECoh{\Gamma}{\P^2}$, as we will show in this section.

	\begin{proposition}\label{prop:morita}
		The following pairs of algebras are Morita equivalent:
		
		\begin{itemize}
			\item $\Pi$ and $\C[x,y]\# \Gamma$,
			\item $\Pibullet$ and $\C[x,y,z]\# \Gamma$,
			\item $\Pi_0 = \Pibullet_0$ and $\C \Gamma$.
		\end{itemize}
		
		Furthermore, these Morita equivalences are compatible in the sense that the Morita equivalence induced between $ \Pibullet/z\Pibullet $ and $(\C[x,y,z]\# \Gamma)/(z\C[x,y,z]\#\Gamma) $ (recall our notation from \cref{not:zElement}) by the second equivalence in the list above agrees with the first. 
	\end{proposition}
	
	\begin{proof}
		The first statement is \cite[Proposition 2.3]{CGGS2}. We recall the construction: for every $\rho_i$ there is an idempotent $f_i\in \C\Gamma$ such that $\C\Gamma f_i\isoto \rho_i$. Letting $f = \sum f_i$, there is an isomorphism $f(\C[x,y]\# \Gamma)f \isoto \Pi $, and and the categorical equivalence of modules is induced by mapping a $\C[x,y]\# \Gamma$-module $M$ to $fM$.
		The second statement is \cite[Proposition 6.0.1]{BGK}, and is proved in exactly the same fashion.
		The third follows from noting that the two above Morita equivalences preserve degree. Finally, the indicated compatibility follows from noting that $\Gamma$ acts trivially on $z$.
		
	\end{proof}
	\begin{lemma}\label{lem:UncorneredIsStrongReg}
		The graded preprojective algebra $\Pibullet$ is strongly regular (see \cite[p.2]{BGK}), of global dimension 2 and Gorenstein parameter $-3$.
	\end{lemma}
	\begin{proof}
		It is shown in \cite[Proposition 7.3.4, Proposition 7.2.3]{BGK} that $\C[x,y,z]\#\Gamma$ is strongly regular of global dimension 2 and Gorenstein parameter $-3$, but this algebra is Morita equivalent to $\Pibullet$ by \cref{prop:morita}.
	\end{proof}
	
	\begin{lemma}\label{lem:EquivCategories}
		There are natural equivalences
		\begin{itemize}
			\item $\iota_1\colon\ECoh{\Gamma}{\P^2}\isoto \Coh{\Proj \Pibullet} $
			\item $\iota_2\colon\ECoh{\Gamma}{\P^1}\isoto \Coh{\Proj \Pi} $
			
		\end{itemize}
		with the property that if $\shf F\in \ECoh{\Gamma}{\P^2}$, then $\shf F|_{L} \cong \iota_2^{-1}(\iota_1(\shf F)/z\iota_1(\shf F))$, where we are taking $L\cong \P^1$ to be the line $ \{z = 0\}\subset \P^2$.
	\end{lemma}
	
	\begin{proof}
		The first equivalence is a special case of \cite[Proposition 6.0.1]{BGK}, and the second can be proved \emph{mutatis mutandis}, using \cref{prop:morita}. The final statement also follows from \cref{prop:morita}. 
	\end{proof}
	\subsection{Cornering}

	\begin{definition}\label{def:corneredAlgs}
		Choose a nonempty $I\subset Q_0$, and set $e_I = \sum_{i\in I} e_i$.
		
		We define \[\Pibullet_I \coloneqq e_I\Pibullet e_I,\] and similarly define $\Pi_I \coloneqq e_I\Pi e_I $. 
	\end{definition}
	We fix $I$ for the rest of the paper.
	
	We call the algebras $\Pibullet_I, \Pi_I$ \emph{cornered}, and the process of passing from $\Pibullet$ (for example) to $e_I \Pibullet e_I$ \emph{cornering}.
	Note also that we have $\Pibullet_I/z_I\Pibullet_I  = \Pi_I$.
	
	There are functors $j_!\colon \Pibullet_I\cat{-mod}\to \Pibullet\cat{-mod},\quad j^*\colon \Pibullet\cat{-mod}\to \Pibullet_I\cat{-mod}$, defined by 
	\[j_!(M) = \Pibullet e_I\otimes_{\Pi^\bullet_I} M, \quad j^*(N) = e_I\Pibullet\otimes_{\Pibullet} N.\]
	These form two out of the six functors of a \emph{recollement} (\cite{CIK18}), and have the following properties:
	\begin{lemma}\label{lem:RecollementFunctors}
		\begin{itemize}
			\item $j^*$ is exact;
			\item $j_!$ is fully faithful;
			\item $j_!$ is the left adjoint of $j^*$ (and therefore right exact);
			\item given a $\Pi_I$-module $M$, 
			$\Ext^k_{\Pi^\bullet}(j_!M, N) = 0$ for $k= 0,1$\footnote{In \cite{CIK18}, it is claimed that this holds for \emph{every} $k$, but this is a mistake.} and every $\Pibullet/\Pibullet e_I\Pibullet$-module $N$;
			\item $j^*j_! = \id$.	
		\end{itemize}
	\end{lemma}
	\begin{proof}
		The first four claims are all from \cite[Section~3]{CIK18}. 
		The fifth claim holds because \[j^*j_! M = e_I\Pibullet\otimes_{\Pibullet}\Pibullet e_I\otimes_{\Pibullet_I} M = \Pibullet_I\otimes_{\Pibullet_I} M = M.\]
	\end{proof}
	
	\begin{lemma}\label{lem:CornerFunctorsPreserveFinDim}
		The functors $j^*$, $j_!$ take finite-dimensional modules to finite-dimensional modules.
	\end{lemma}
	\begin{proof}
		This is straightforward for $j^*$ -- as a vector space, we have $j^*M = e_IM \subset M$ for any $\Pibullet$-module $M$. For $j_!$, we can repeat the proof of \cite[Lemma 3.6]{CIK18}.
	\end{proof}
	
	We will now show that the noncommutative algebra $\Pi^\bullet_I$ has fairly nice properties. We shall not strictly need them for our main results, but we still find them worth registering.
	
	\begin{proposition}
		The algebra $\Pi^\bullet_I$ is noetherian.
	\end{proposition}
	\begin{proof}
		We show that $\Pi^\bullet_I$ is left noetherian, right-noetherianness is shown similarly.
		
		Let $0\subset I_1\subset I_2\subset\cdots$ be a chain of left $\Pi^\bullet_I$-ideals, thus \begin{equation}\label{eq:firstChain}
			\Pi^\bullet_I\to \Pi^\bullet_I/I_1\to \Pi^\bullet_I/I_2\to \cdots
		\end{equation} is a chain of surjective $\Pi^\bullet_I$-module homomorphisms. Then, because $j_!$ is right exact, $j_!\Pi^\bullet_I\to j_!(\Pi^\bullet_I/I_1)\to j_!(\Pi^\bullet_I/I_2)\to \cdots$  is a chain of surjective $\Pibullet$-module hommoorphisms. But we have a surjective homomorphism $\Pibullet\to \Pibullet e_I = j_!\Pibullet_I$, simply given by right-multiplying by $e_I$. Composing with this, we obtain a chain of surjective homomorphisms 
		\begin{equation}\label{eq:secondChain}
			\Pibullet\to\Pibullet e_I\to j_!(\Pi^\bullet_I/I_1)\to j_!(\Pi^\bullet_I/I_2)\to \cdots,\end{equation}
		and by the noetherianness of $\Pibullet$, \eqref{eq:secondChain} must terminate, let us say at $j_!(\Pibullet_I/I_n)$. Apply $j^*$ to this chain, we find (since $j^*j_! = \id$) that \eqref{eq:firstChain} terminates at $\Pibullet_I/I_n $. But then the original chain of ideals also terminates at $I_n$, so $\Pibullet_I$ is left noetherian.
	\end{proof}
	
	\begin{proposition}\label{prop:piBulletIProperties}
		The algebra $\Pibullet_I$ has the following properties:
		\begin{itemize}
			\item $\Pibullet_{I,0}$ is semisimple,
			\item $\Pibullet_I$ has polynomial growth \ie there are integers $m_1, m_2$ such that \[\dim {\Pibullet_{I, k}}<m_1k^{m_2}\quad \textrm{for all }k,\]
		\end{itemize}
	\end{proposition}
	\todo{add something about ampleness}
	\begin{proof}
		
		The first claim follows from the fact that $\Pi^\bullet_{I,0}$ is a direct sum of 1-dimensional algebras. The second claim follows from the inclusion of algebras $\Pi^\bullet_I\subset \Pi^\bullet$ and the fact \cite{BGK} that $\Pibullet$ has polynomial growth. 
		
	\end{proof}
	
	We also note here the following technical lemma.
	
	\begin{proposition}\label{prop:LongPathGoesEverywhere}
		There is a number $n$ depending on $\Gamma$ and $I$ such that any homogeneous class $p$ in $\Pi$ of length greater than $n$ is equivalent to the class of a path passing through $I$.
	\end{proposition}
	\begin{proof}
		Let us first note that the algebra $\Pi/\Pi e_I \Pi$ is finite-dimensional.
		This follows from \cite[Proposition 2.1]{SavTing}, because deleting a nonempty set of vertices and all their incident edges from an affine Dynkin diagram results in a disjoint union of finite Dynkin diagrams. Thus, in the homomorphism $q\colon \Pi\to \Pi/\Pi e_I \Pi$, there is an integer $m$ such that the class of any path $p\in \Pi$ of length greater than $m$ is mapped to 0 by $q$. But this is what we wished to show.
	\end{proof}

	\section{The geometry of $\P^2_I$}\label{sec:CorneredP2Geometry}
	We can now define the main object of our research. 
	
	\begin{definition}\label{def:ProjPartRes}
		We set
		$\P^2_{I} \coloneqq \Proj \Pibullet_I$.
	\end{definition}
	Note that this really means that we are setting $\Coh{\P^2_{I}} = \cat{qgr}(\Pibullet_I)$.
	
	Similarly, set $\P^1_I \coloneqq \Proj \Pi_I$.
	
	\begin{notation}
		We shall write $\P^2_\Gamma$ for $\P^2_{Q_0} = \Proj \Pibullet$, and similarly $\P^1_\Gamma$ for $\P^1_{Q_0} = \Proj \Pi$. 
	\end{notation}
	\begin{remark}
		This is done to make our notation agree for these spaces with that of \cite{BGK}. In \emph{ibid}, these spaces are defined as (respectively) \[\P^2_\Gamma\coloneqq \Proj \C[x,y,z]\#\Gamma,\quad \P^1_\Gamma\coloneqq\Proj \C[x,y]\#\Gamma,\] but by \cref{prop:morita} and \cref{lem:EquivCategories} the resulting categories of coherent sheaves are naturally equivalent. In this context, one can also use the notation $[\mathbb{P}^2/{\Gamma}]$ for $\P^2_\Gamma$.
	\end{remark}
	\begin{remarks}
		\begin{enumerate}
			\item Instead of defining $\P^2_I$ using subalgebras of $\Pibullet$, we could have used subalgebras of $\C[x,y,z]\#\Gamma$. Let $f_i\in \C\Gamma$ be the idempotent such that $\C\Gamma f_i \isoto \rho_i$ as a $\Gamma$-representation. One can then show, using the Morita equivalences of \cref{prop:morita}, that $\Coh{\Proj \Pibullet_I} = \Coh{\Proj f_I(\C[x,y,z]\#\Gamma) f_I}$, where $f_I = \sum_{i\in I} f_i$. We choose, however, to use subalgebras of $\Pibullet$, because this makes it easier to define the functors $\tau_*, \tau^*$, and also because the degree-0 subalgebra $\Pibullet_{I,0}$ is a direct sum of 1-dimensional algebras, which we find simpler to work with than $\C\Gamma$, the degree-0 subalgebra of $\C[x,y,z]\#\Gamma$.
			
			\item We are not aware whether the noncommutative surfaces $\P^2_I$ have already appeared in the literature; a cursory survey suggests they may not have. 
			It seems, for instance, that they are not among the surfaces considered in 
			\cite{Sierra}: there it is required for surfaces $\Proj R$ to satisfy $R_0$ being a field. This is not \emph{a priori} true for $\P^2_I$, except for the case where $I$ is a singleton. 
		\end{enumerate}
	\end{remarks}
	
	From \cref{lem:EquivCategories}, we see that $\P^2_\Gamma\isoto [\P^2/\Gamma]$, in the sense that there is a natural equivalence of categories $\Coh{\P^2_\Gamma}\isoto \Coh{[\P^2/\Gamma]}$. We can describe at least one other of the spaces $\P^2_I$.
	\begin{proposition}\label{prop:CorneredTo0}
		We have $\Pibullet_{\{0\}} = \C[x,y,z]^\Gamma$. 
	\end{proposition}
	\begin{proof}
		Consider the algebra $\Pibullet_{\{0\}}/(z_0) = \Pi_{\{0\}}$. By the construction in \cite[Section 6]{CGGS}, $\Pi_{\{0\}}\isoto\C[x,y]^\Gamma$. Thus $\Pibullet_{\{0\}}/(z_0)$ is the quotient of the path algebra of a quiver with four loops at a single vertex, three of which correspond to the generators of $\C[x,y]^\Gamma$. Because the class of the $z_0$-loop commutes with every other class in $\Pibullet_{\{0\}}$, this proves that $\Pibullet_{\{0\}} = \C[x,y,z]^\Gamma$.
	\end{proof}
	
	We can now prove the first of the Theorems announced in the introduction.
	\begin{proof}[Proof of \cref{thm:first}]
		This now follows from \cref{lem:EquivCategories} and Proposition~\ref{prop:CorneredTo0}.
	\end{proof}
	
	\begin{definition}
		Define the functor $c^*\colon\Coh{\P^2_{I}}\to \Coh{\P^1_I}$ `restriction to the line at infinity' to be the functor induced by the following functor on the level of modules: $M\in \Pibullet_I\textrm{-}\cat{mod}\mapsto M/zM\in \Pi_I\textrm{-}\cat{mod}$. It is clear that $M/zM$ is finite-dimensional if $M$ is, so $c^*$ is well-defined.
	\end{definition}
	We see that for $I = Q_0$, this $c^*$ agrees with that already defined following \cref{lem:EmbeddingSkewGroup} (when the Morita equivalences of \cref{prop:morita} are applied).
	
	Consider the functor $j_!\colon \Pibullet_I\textrm{-}\cat{mod}\to \Pibullet-\cat{mod}$. Let $\shf F\in \Coh{\P^2_I}$, and let $M\in \Pibullet_I\textrm{-}\cat{mod}$ be a module such that $\pi(M) = \shf F$. Set then
	\[\tau^*\shf F = \pi(j_!M).\] This is well-defined by \cref{lem:CornerFunctorsPreserveFinDim}. Similarly, let $\shf E\in \Coh{\P^2_\Gamma}$, and let $N\in \Pibullet\textrm{-}\cat{mod}$ be a module such that $\pi(N) = \shf E$. Set then
	\[\tau_*\shf E = \pi(j^* N),\] which is also well-defined by \cref{lem:CornerFunctorsPreserveFinDim}.
	
	\begin{definition}\label{def:TauFunctors}
		The above discussion defines two functors \[\tau^*\colon \Coh{\P^2_I}\to \Coh{\P^2_\Gamma},\quad \tau_*\colon \Coh{\P^2_\Gamma}\to \Coh{\P^2_I}.\]
	\end{definition}
	We will make heavy use of the two above functors. They inherit several properties from $j_!, j^*$.
	
	\begin{notation}
		Given $I, J\subset Q_0$, we will write $\OO_{\P^2_\Gamma}e_I$ for the sheaf $\pi(\Pibullet e_I)$ on $\P^2_\Gamma$, and similarly $\OO_{\P^2_I}e_J$ for the sheaf $\pi(e_I\Pibullet e_J) = \tau_*(\OO_{\P^2_\Gamma}e_J)$.
	\end{notation}
	
	\begin{proposition}\label{prop:TauFunctorsAreOK}
		The following properties hold:
		\begin{enumerate}
			
			\item $\tau^*$ is left adjoint of $\tau_*$,
			\item $\tau_*$ is exact,
		\end{enumerate}
	\end{proposition}
	
	\begin{proof}
		\textbf{(1):}
		Let $\shf F\in \Coh{\P^2_I}$, represented by some $\Pibullet_I$-module $M$, and let $\shf G\in \Coh{\P^2_\Gamma}$, represented by some $\Pibullet$-module $N$. We then have:
		\[ \Hom_{\P^2_I}(\shf F, \tau_*(\shf G)) = \varinjlim\Hom_{\Pibullet_I}(M', (j^*N)/N') ,\] where the limit ranges over all pairs $(M', N')$ of submodules of $M, j^*N$ respectively, with the property that $M/M', N'$ are finite-dimensional.
		
		Every $\Pibullet_I$-module $N'$ can be written as $j^*j_!(N')$, which furthermore demonstrates that every finite-dimensional submodule of $j^*N$ is the image under $j^*$ of a finite-dimensional submodule of $N$ (indeed, if $N'\subset j^*N$, let $N_1$ be the image of $j_!(N')$ under the composed morphism $j_!(N')\to j_!j^*(N)\to N$, then $j^*(N_1) = j^*j_!(N') = N'$) and we can write 
		\[ \varinjlim\Hom_{\Pibullet_I}\left(M', (j^*N)/N'\right) = \varinjlim\Hom_{\Pibullet_I}\left(M', (j^*N)/j^*(\tilde N)\right) ,\] with the limit running over pairs $M'\subset M, \tilde{N}\subset N$, again with $M/M'$ and $\tilde{N}$ finite-dimensional. Now, using adjointness of $j^*$ and $j_!$, this becomes
		\[ \varinjlim\Hom_{\Pibullet_I}\left(M', (j^*N)/j^*(\tilde N)\right) = \varinjlim\Hom_{\Pibullet_I}\left(M', j^*(N/\tilde N)\right) =  \varinjlim\Hom_{\Pibullet}\left(j_!M', N/\tilde N\right).\] For a fixed $M', \tilde{N}$, let $\phi\in \Hom_{\Pibullet}\left(j_!M', N/\tilde N\right)$. Let then the image of the morphism $j_!(M')\to j_!(M)$ be denoted as $\tilde{M}$, we then have a canonical commutative diagram
		
		\[ \begin{tikzcd}
			\ker p \arrow[r, hook] \arrow[d] & j_!(M') \arrow[r, "p", two heads] \arrow[d, "\phi"] & \tilde M \arrow[r, "i", hook] \arrow[d, "\tilde \phi", dashed] & j_!(M) \\
			\phi(\ker p) \arrow[r, hook]     & N/\tilde N \arrow[r]                                & (N/\tilde N)/\phi(\ker p)                                      &       
		\end{tikzcd} \] inducing the morphism $\phi$. Now $\coker i = j_!(M/M')$ by right-adjointness of $j_!$, and is thus finite-dimensional. On the other hand, by definition of $j_!$, $\ker p $ is a quotient of $\Tor^{\Pibullet_I}_1(\Pibullet e_I, M/M')$, which is finite-dimensional because $M/M'$ is. This implies that $\phi(\ker p)$ is finite-dimensional, and so $\tilde{\phi}$ induces an element of 
		\begin{equation}\label{eq:defSheafHom} \varinjlim\Hom_{\Pibullet}\left(\tilde{M}, N/ \tilde N\right) = \Hom_{\P^2_\Gamma}(\tau^*\shf F, \shf G) ,(M)\end{equation} again with the limit taken over all pairs  $(\tilde{M}\subset j_!M, \tilde{N}\subset N)$ with $j_!M/\tilde{M}, \tilde{N}$ finite-dimensional.
		
		To pass in the other direction, choose a pair $\tilde{M}\subset j_!(M), \tilde{N}\subset N$ in \eqref{eq:defSheafHom}, and let $\psi\in \Hom_{\Pibullet}\left(\tilde{M}, N/ \tilde N\right)$.  Precomposing with $j_!j^* $  defines a morphism \[\psi'\in \Hom_{\Pibullet}\left(j_!j^*\tilde{M}, N/ \tilde N\right) = \Hom_{\Pibullet_I}\left(j^*\tilde{M}, j^*(N/ \tilde N)\right) = \Hom_{\Pibullet_I}\left(j^*\tilde{M}, j^*(N)/ j^*(\tilde N))\right)\] where we have used the adjointness of $j_!, j^*$ again. But because $j^*$ is exact and takes finite-dimensional modules to finite-dimensional modules, this means that $\psi'$ defines an element of 
		\[ \varinjlim\Hom_{\Pibullet_I}\left(M^\dagger, (j^*N)/N^\dagger\right) ,\]  the limit taken over all pairs $M^\dagger\subset M, N^\dagger\subset j^*(N)$ with $N^\dagger, M/M^\dagger$ finite-dimensional. But this latter limit is by definition \[ \Hom_{\P^2_I}( \shf F, \tau_*\shf G) .\] It is simple to show that the two constructions just given are inverse to one another, and thus we have shown that $\tau^*, \tau_*$ are indeed adjoint. This proves point (1).
		
		\textbf{(2):}
		The adjunction immediately implies that $\tau^*$ is right exact, $\tau_*$ left exact. It thus remains to show that $\tau_*$ is right exact. So let $\shf F, \shf G \in \Coh{\P^2_{\Gamma}}$, and assume that $\phi\in \Hom_{\P^2_\Gamma}(\shf F, \shf G)$ is surjective. Let $M, N$ be two $\Pi^\bullet$-modules representing $\shf F, \shf G$ respectively. Without loss of generality, we can assume that $\phi$ can be represented by a morphism $\psi\colon M\to N$ with finite-dimensional cokernel. But then $\tau_*(\phi)$ is represented by $j^*(\psi)$, which is -- again since $j^*$ takes finite-dimensional modules to finite-dimensional modules -- a module homomorphism with finite-dimensional cokernel. Thus $\tau_*\phi$ is surjective, and $\tau_*$ is right exact. This proves point (2).
		
	\end{proof}
	We also conjecture that $\tau^*$ is fully faithful. We are unable to show it, and it will not be necessary.

	The following lemma is not needed for our main results, but we include it as an additional example of how the functors $\tau_*, \tau^*$ are natural.
	\begin{lemma}\label{lem:InternalDuals} 
		
		The internal dualisation functors commute with $\tau^*$\ie if $\shf F\in \Coh{\P^2_{I}}$, then 
		$(\tau^*\shf F)\dual = \tau^*_{\mathrm{right}}(\shf F\dual)$. 
	\end{lemma}
	\begin{proof}
		We show this on the level of modules. If $M$ is a $\Pibullet_I$-module representing $\shf F$, then we wish to show that:
		\[ (j_!)_{\mathrm{right}}\left(\bigoplus_{k\ge 0}(\Hom_{\Pibullet_I}(M, \Pibullet_I(k)))\right) =  \bigoplus_{k\ge 0}\Hom_{\Pibullet}(j_!M, \Pibullet(k)).\]
		We find:
		\[	(j_!)_{\mathrm{right}}\left(\bigoplus_{k\ge 0}\Hom_{\Pibullet_I}(M, \Pibullet_I(k))\right) =\left(\bigoplus_{k\ge 0}\Hom_{\Pibullet_I}(M, \Pibullet_I(k))\right)\otimes_{\Pibullet_I} e_I\Pibullet \]
		\[ = \bigoplus_{k\ge 0}\Hom_{\Pibullet_I}(M, e_I\Pibullet(k)) = \bigoplus_{k\ge 0}\Hom_{\Pibullet_I}(M, j^*\Pibullet(k)) \]
		\[ = \bigoplus_{k\ge 0}\Hom_{\Pibullet}(j_!M, \Pibullet(k))\] as wanted.
		
	\end{proof}
	
	\subsection{Computing $\Ext$-groups}

	\begin{proposition}\label{prop:ExtComputing}
		Let $\shf F\in \Coh{\P^2_I}$. Then $\Ext^1_{\P^2_I}(\OO_{\P^2_I}, \shf F) \cong \Ext^1_{\P^2_\Gamma}(\tau^*\OO_{\P^2_I}, \tau^*\shf F) $.
	\end{proposition}
	
	\begin{proof}
		We start by noting that $\Ext^1_{\P^2_I}(\OO_{\P^2_I}, \shf F)$ classifies short exact sequences
		\[ 0\to \shf F \to \shf G \to  \OO_{\P^2_I}\to 0.\] Applying $\tau^*$ to such a sequence gives 
		\[ \cdots\to L_1\tau^* \OO_{\P^2_I}\to \tau^*\shf F\to \tau^*\shf G\to \tau^* \OO_{\P^2_I}\to 0 .\]
		We now show that $L_1\tau^* \OO_{\P^2_I} = 0$.
		
		To see this, note that the Serre quotient functor $\pi\colon \cat{gr}(\Pibullet_I)\to \cat{qgr}(\Pibullet_I) =\Coh{\P^2_I}$ is exact, so we have $\pi(L_1j_!(\Pibullet_I)) = L_1\tau^*\OO_{\P^2_I}$. But this vanishes because $L_1j_!(\Pibullet_I) = \Tor_1^{\Pibullet_I}(\Pibullet e_I, \Pibullet_I) = 0 $.
		
		We thus have a map \[ \phi\colon \Ext^1_{\P^2_I}(\OO_{\P^2_I}, \shf F)\to \Ext^1_{\P^2_\Gamma}(\tau^*\OO_{\P^2_I}, \tau^*\shf F). \] Consider now an extension
		\[ 0\to \tau^*\shf F\to \shf E\to \tau^* \OO_{\P^2_I}\to 0, \] and note that we get a commutative diagram with exact rows
		
		\[\begin{tikzcd}
			0 \arrow[r] & \tau^*\shf F \arrow[r] \arrow[d, "\cong"] & \tau^*\tau_*\shf E \arrow[d] \arrow[r] & \tau^*\OO_{\P^2_I} \arrow[d, "\cong"] \arrow[r] & 0 \\
			0 \arrow[r] & \tau^*\shf F \arrow[r]           & \shf E \arrow[r]                       & \tau^*\OO_{\P^2_I} \arrow[r]           & 0
		\end{tikzcd}\] implying that $\tau^*\tau_*\shf E \cong \shf E$. This implies that the map \[\psi\colon \Ext^1_{\P^2_\Gamma}(\tau^*\OO_{\P^2_I}, \tau^*\shf F)\to\Ext^1_{\P^2_I}(\OO_{\tau_*\tau^*\P^2_I}, \tau_*\tau^*\shf F) = \Ext^1_{\P^2_I}(\OO_{\P^2_I}, \shf F) \]	induced by exactness of $\tau_*$ is a double-sided inverse of $\phi$.
	\end{proof}
	
	\begin{proposition}\label{prop:realExting}
		Let $\shf F\in \Coh{\P^2_I}$. Then 
		\begin{enumerate}
			\item $H^1({\P^2_I},\shf F(-1)) = e_IH^1({\P^2_\Gamma},\tau^*\shf F(-1))$.
			\item Let $\shf E\in \Coh{\P^2_\Gamma}$ be such that $\tau_*\shf E = \shf F$. Then $H^1({\P^2_I},\shf F(-1)) = e_IH^1({\P^2_\Gamma},\shf E(-1))$.
		\end{enumerate}
	\end{proposition}
	
	\begin{proof}
		The first part follows from \cref*{prop:ExtComputing}: We have
		\begin{align*} &H^1({\P^2_I},\shf F(-1))  = \Ext^1_{\P^2_I}(\OO_{\P^2_I}, \shf F(-1))= \\ 
			= &\Ext^1_{\P^2_\Gamma}(\tau^*\OO_{\P^2_I}, \tau^*\shf F(-1)) = \Ext^1_{\P^2_\Gamma}(\OO_{\P^2_\Gamma}e_I, \tau^*\shf F(-1)) = \\ =&e_I\Ext^1_{\P^2_\Gamma}(\OO_{\P^2_\Gamma}, \tau^*\shf F(-1))  =e_IH^1({\P^2_\Gamma},\tau^*\shf F(-1)). \end{align*}
		
		For the second, note that it will be enough to show that \[  e_IH^1({\P^2_\Gamma},\shf E(-1)) =  e_IH^1({\P^2_\Gamma},\tau^*\tau_*\shf E(-1)) .\]
		To do this, let $M$ be a $\Pibullet_I$-module representing $\shf E(-1)$. It follows that $j_!j^* M$ represents $\tau^*\tau_*\shf E(-1)$. Let $c\colon j_!j^*M\to M$ be the induced map, then the exact sequence
		\[ 0\to \ker c\to j_!j^*M\to M \to \coker c\to 0\] represents
		\[ 0\to \pi(\ker c)\to \tau^*\tau_*\shf E(-1)\to \shf E(-1)\to \pi(\coker c)\to 0. \] Note especially that we must have $e_I\ker c = e_I\coker c = 0$.
		
		So we are reduced to showing that  $e_IH^1({\P^2_\Gamma},\pi(\ker c)) = e_IH^1({\P^2_\Gamma},\pi(\coker c)) = 0$.
		
		In order to do this, recall that taking direct limits is exact, thus we have for any two $\Pibullet_I$-modules $M, N$:
		
		\[ \Ext^1_{\cat{qgr}(\Pibullet_I)}(M, N) = \varinjlim \Ext^1_{\cat{gr}(\Pibullet_I)}(M', N/N'), \] as always, the limit ranging over $M'\subset M, N'\subset N$ such that $M/M', N'$ are finite-dimensional. It follows that 
		\[ \Ext^1_{\Coh{\P^2_\Gamma}}(\OO_{\P^2_\Gamma}, \pi(\ker c)) = \varinjlim\Ext^1_{\cat{gr}(\Pibullet_I)}(M', \ker c/N') ,\] where $M'\subset \Pibullet_I, N'\subset \ker c$. But $e_I\ker c = 0$ implies $e_I (\ker c/N') = 0$, so 
		\[e_IH^1({\P^2_\Gamma},\pi(\ker c)) =  e_I\varinjlim\Ext^1_{\cat{gr}(\Pibullet_I)}(M', \ker c/N') = \varinjlim\Ext^1_{\cat{gr}(\Pibullet_I)}(M', e_I\ker c/N') = 0. \] The same argument works for $\coker c$, and we are done.
	\end{proof}
	
	\subsection{Torsion-free and locally free sheaves}
	
	\begin{definition}
		We say that a sheaf $\shf F\in \Coh{\P^2_I}$ is \emph{torsion-free} if there is a torsion-free sheaf $\shf G\in \Coh{\P^2_\Gamma}$ such that $\tau_*\shf G \isoto \shf F$.
		
		We say that a sheaf $\shf F\in \Coh{\P^2_I}$ is \emph{locally free} if $\tau^*\shf F$ is locally free.
	\end{definition}
	It need not be true that $\tau^*$ takes a torsion-free sheaf to a torsion-free sheaf. In \cref{subsec:tfinvimg} below we will introduce a modified version of $\tau^*$ that handles this.
	
	We will need the following definition in \cref{sec:connections}. It is adapted from \cite[Definition 3.3.12, Lemma 3.3.13]{BGK}.
	\begin{definition}
		A sheaf $\shf F\in \Coh{\P^2_{I}}$ is $z$-\emph{torsion free} if the morphism $(z_I\cdot)\colon \shf F\to \shf F(1)$ is a monomorphism.
	\end{definition}

	\begin{lemma}\label{lem:halfCorneredLocFree}
		The sheaves $\OO_{\P^2_\Gamma}e_I(k) = \tau^*\OO_{\P^2_I}\in \Coh{\P^2_{\Gamma}}$ are locally free. 
	\end{lemma}
	\begin{proof}
		This is simply because on the module level $ \Pibullet e_I $ is a direct summand of $\Pibullet$, so $\OO_{\P^2_\Gamma}e_I(k)$ is a direct summand of $\OO_{\P^2_\Gamma}(k)$, which is locally free, and so the higher Ext-groups vanish as required.
	\end{proof}

	\section{Sheaves on $\P^2_I$ and Nakajima quiver varieties}\label{sec:MoveToNQV}
	\subsection{Framing}
	\label{subsec:framing}
	
	\begin{definition}\label{def:framingOnNCPR}
		
		Let $\shf F\in \Coh{\P^2_{I}}$, and fix a finite-dimensional $\Pi$-module $W$. Then $W$ corresponds to a finite-dimensional $\Gamma$-representation $\widehat{W}= \bigoplus_{i\in Q_0} W_i\otimes \rho_i$. We define a \emph{$W$-framing} of $\shf F$ to be an isomorphism \[\phi_{\shf F}\colon c^*\shf F\isoto \pi(e_I\Pi\otimes_{\Pi_0}W ) = e_I\OO_{\P^1_\Gamma}\otimes W.\]
		The pair $(\shf F, \phi_{\shf F})$ will be called a \emph{$W$-framed sheaf} or just a \emph{framed sheaf}, and we will often suppress $\phi_{\shf F}$ from notation. A morphism of two $W$-framed sheaves $(\shf F, \phi_{\shf F})$, $(\shf G, \phi_{\shf G})$ is a morphism $t\colon\shf F \to \shf G$ such that $\phi_{\shf F} = \phi_{\shf G}\circ t$. We denote the category of $W$-framed sheaves on $\P^2_{I}$ by \[\Cohfr{\P^2_{I}}(W).\]
	\end{definition}
	
	\begin{lemma}\label{lem:FramingCanBeCornered}
		If $\shf F$ is a $W$-framed sheaf on $\P^2_{I}$, and $W_i= 0$ for all $i\not\in I$, then giving $\phi_{\shf F}$ is equivalent to giving an isomorphism $c^*\shf F\isoto \Pi_I\otimes_{{\Pi_I}_0}W$.
	\end{lemma}
	\begin{proof}
		Let $I^c \coloneqq Q_0\setminus I$.
		Then this is an easy computation: We have \[e_I\Pi\otimes_{\Pi_0}W = \left(e_I\Pi e_I\otimes_{{\Pi_{I, 0}}}(\bigoplus _{i\in I}W_i)\right)\oplus \left(e_I\Pi e_{I^c}\otimes_{\Pi_{I^c,0}}(\bigoplus_{i\in I^c}W_i) \right),\] and if $\bigoplus_{i\in I^c}W_i =0$, the second summand vanishes.
	\end{proof}
	
	\begin{proposition}\label{prop:PreservesFraming}
		Fix $W$ as above. The functors $\tau^*, \tau_*$ induce functors (which we write with the same notation) \[\tau_*\colon\Cohfr{\P^2_\Gamma}(W)\to \Cohfr{\P^2_{I}}(W),\quad \tau^*\colon \Cohfr{\P^2_{I}}(W)\to \Cohfr{\P^2_\Gamma}(W),\] and we have $\tau_*\tau^* = \id$.
		
	\end{proposition}
	\begin{proof}
		Let $\shf F$ be a $W$-framed sheaf on $\P^2_\Gamma$. 
		We have a short exact sequence
		\[ 0\to z\shf F\to \shf F \to \shf F/z\shf F\isoto \OO_{\P^1_\Gamma}\otimes W\to 0, \]
		and because $\tau_*$ is an exact functor, this gives \[ e_I\OO_{\P^1_\Gamma}\otimes W\isoto\tau_*(\shf F/z\shf F)\isoto \tau_*(\shf F)/\tau_*(z\shf F) = \tau_*(\shf F)/z_I\tau_*\shf F, \] as desired.
		
		In the other direction, let $\shf E$ be a $W$-framed sheaf on $\P^2_{I}$. 
		There is then a short exact sequence
		\[ 0\to z_I\shf E\to \shf E\to \shf E/z_I\shf E\isoto e_I\OO_{\P^1_{\Gamma}}\otimes W\to 0. \]
		
		Applying $\tau^*$, we obtain:
		\[  \tau^*(z_I\shf E)\overset{a}{\to} \tau^*\shf E\to \tau^*(\shf E/z_I\shf E)\isoto \tau^*(e_I\OO_{\P^1_{\Gamma}}\otimes W)\to 0. \]
		We first show that $a$ is injective: If we let $N$ be a $\Pibullet_I$-module representing $\shf E$, we have: \[ \tau^*(z_I\shf E) = \pi(j_!(z_IN)).\]  But \[j_!(z_IN) =\Pi^\bullet e_I\otimes_{\Pi^\bullet_I}z_IN = z\Pi^\bullet e_I\otimes_{\Pi^\bullet_I}N = zj_!(N) ,\] which is clearly a submodule of $j_!(N) = \Pibullet e_I\otimes_{\Pibullet_I}N$, so $a$ is injective.

		Consider then $\tau^*(e_I\OO_{\P^1_{\Gamma}}\otimes W) = \pi(\Pi^{\bullet}e_I\otimes_{\Pi^\bullet_I}e_I\Pi\otimes W) $. 
		We note that $\Pi^{\bullet}e_I\otimes_{\Pi^\bullet_I}e_I\Pi = \Pi e_I\Pi$. This is a subalgebra of $\Pi$. Consider the length filtration on $\Pi$: By \cref{prop:LongPathGoesEverywhere}, there is an integer $p$ such that any class of length $>p$ can be written as a sum of paths each passing through a vertex in $I$. But the space of elements of length $\le p$ is finite-dimensional. Thus $\Pi/\Pi e_I\Pi$ is finite-dimensional, and $\Pi^{\bullet}e_I\Pi\otimes_{\Pi_0} W$ and $\Pi\otimes_{\Pi_0} W$ represent the same element of $\Coh{\P^1_\Gamma}$. This concludes the proof.
	\end{proof}
	\subsection{Torsion-free inverse images}
	\label{subsec:tfinvimg}
	As mentioned, the functor $\tau^*$ may not preserve torsion-freeness. Let $\shf K = \pi(\bigoplus_{k\ge 0} K_k),\quad \shf L=\pi(\bigoplus_{k\ge 0} L_k)$ be subsheaves of $\shf G = \pi(\bigoplus_{k\ge 0} G_k) \in \Coh{\P^2_\Gamma}$, and assume that $c^*\shf K = c^*\shf L = 0$. On the level of modules, this means that there is a $k'\ge 0$ such that for all $k\ge k'$ we have
	\begin{enumerate}
		\item $K_k\subset G_k,\quad L_k\subset G_k$, and
		\item $zK_k = K_{k+1},\quad zL_k = L_{k+1}$.
	\end{enumerate}
	
	We can then set $\shf K +\shf L = \pi(\bigoplus_{k\ge k'} K_k+L_k)$, where the sums are taken as submodules of $G_k$. It is then clear that $z(K_k+L_k) = K_{k+1} + L_{k+1}$ for all $k\ge k'$, and so $c^*(\shf K +\shf L) = 0 $. It follows that there is a unique maximal subsheaf $\shf T\subset \shf G$ such that $c^*\shf T =0$.
	
	\begin{definition}\label{def:TorFreeInvIm}
		Let $\shf F\in \Cohfr{\P^2_I}(W)$, and assume that $\shf F$ is torsion-free. Let $\shf T_{\shf F}$ be the maximal subsheaf of $\tau^*\shf F$ such that $c^*\shf T_{\shf F}=0$, and set \[\tau^T(\shf F)\coloneqq\tau^*\shf F/\shf T_{\shf F}.\] We call $\tau^T\shf F$ the \emph{torsion-free inverse image} of $\shf F$.
	\end{definition}
	The reason for the name is provided by the following lemma:
	
	\begin{lemma}\label{lem:EquivalentTorFreeInvIm}
		Let $\shf F\in \Cohfr{\P^2_I}(W)$. Then 
		$\tau^T\shf F$ is isomorphic to the image of $\tau^*\shf F$ in $(\tau^*\shf F)\dual\dual$, and $\tau^T\shf F$ is torsion-free.
	\end{lemma}
	
	\begin{proof}
		Let $\shf K$ be the kernel of the morphism $\tau^*\shf F\to (\tau^*\shf F)\dual\dual$; we show that $\shf T_{\tau^*\shf F}= \shf K$. First, consider the exact sequence
		\begin{equation}\label{eq:kernelTorsion}
			0\to \shf K\to\tau^*\shf F\to (\tau^*\shf F)\dual\dual,\end{equation} and use the equivalence of categories \cref{lem:EquivCategories} to interpret this as a ($\Gamma$-equivariant) exact sequence of $\Gamma$-equivariant coherent sheaves on $\P^2$. There is then an open subscheme $U\subset \P^2$ on which $\tau^*\shf F$ is locally free, and by arguing as in \cite[Lemma 4.4]{Paper1}, $U$ must contain $\P^1_\Gamma$. Then restrict \eqref{eq:kernelTorsion} to $U$. Because $(\tau^*\shf F)|_U$ is locally free, it follows that $\tau^*\shf F|_U \isoto (\tau^*\shf F\dual\dual)|_U$, and so $\shf K|_U = 0$, so $c^*\shf K = 0$. Thus $\shf K\subset\shf T_{\shf F}$. 
		
		For the other inclusion, note that by \cite[Proposition 2.0.4(5)]{BGK}, $(\tau^*\shf F)\dual\dual$ is locally free, and by \cite[Proposition 3.3.9(10)]{BGK} $\shf T_{\tau^*\shf F}$ is \emph{Artin}. It follows from \cite[Proposition 2.0.9(1)]{BGK} that $\Ext^2_{\P^2_\Gamma}((\tau^*\shf F)\dual\dual(3), \shf T_{\tau^*\shf F}) = 0$, which by Serre duality gives $\Hom_{\P^2_\Gamma}(\shf T_{\tau^*\shf F}, (\tau^*\shf F)\dual\dual)= 0$. But then $\shf T_{\tau^*\shf F}\subset \shf K$, so we have $\shf T_{\tau^*\shf F}=\shf K$.
		
		Finally, $\tau^T\shf F$ is torsion-free because it is a subsheaf of the locally free sheaf $\tau^*\shf F\dual\dual$.
	\end{proof}

	\begin{lemma}\label{lem:EquivalencesOfTorFreeInvIm}
		Let $\shf F\in \Cohfr{\P^2_I}(W)$ be torsion-free. Then $\tau^T\shf F\in \Cohfr{\P^2_\Gamma}(W)$, and $H^1(\P^2_\Gamma, \tau^T\shf F(k)) = H^1(\P^2_\Gamma, \tau^*\shf F(k))$ for all $k$.
	\end{lemma}
	\begin{proof}
		For the first, we know from the proof of the previous Lemma that there is an isomorphism $c^*\tau^*\shf F \isoto c^*(\tau^*\shf F\dual\dual)$, and so we must also have $c^*\tau^T\shf F \cong c^*\tau^*\shf F\cong \OO_{\P^2/\Gamma}\otimes_{\Pi_0}W$. 
		
		As for the claim about cohomology, choose any $k$ and consider the short exact sequence
		\[0\to\shf T_{\shf F}(k)\to\tau^*\shf F(k)\to\tau^T\shf F(k)\to 0, \] giving the long exact sequence 
		\[H^1(\P^2_\Gamma, \shf T_{\shf F}(k))\to H^1(\P^2_\Gamma, \tau^*\shf F(k))\to H^1(\P^2_\Gamma, \tau^T\shf F(k))\to H^2(\P^2_\Gamma, \shf T_{\shf F}(k)).\] By \cite[Proposition 3.3.9(10), Definition 2.0.8]{BGK}, $H^1(\P^2_\Gamma, \shf T_{\shf F}(k))= H^2(\P^2_\Gamma, \shf T_{\shf F}(k))=0$, giving the desired result.
	\end{proof}

	\begin{lemma}\label{lem:TorFreeEquality}
		Let $\shf F\in \Cohfr{\P^2_I}(W)$ be torsion-free. Then $\tau_*\tau^T\shf F = \tau_*\tau^*\shf F = \shf F$.
	\end{lemma}
	\begin{proof}
		
		By definition, there is a torsion-free sheaf $\shf G\in \Coh{\P^2_\Gamma}$ such that $\tau_*\shf G = \shf F$. Consider the induced morphism $	\phi\colon \tau^*\shf F =\tau^*\tau_*\shf G \to \shf G $: because $\shf G$ is torsion-free, $\phi$ must factor through $\tau^T\shf F$. But we have $\tau_*\tau^*\shf F = \tau_*\shf G = \shf F$, so it must hold that $\tau_*\tau^T(\shf F) = \shf F$ as well.
		
	\end{proof}
	
	\begin{lemma}\label{lem:zTorFree}
		Any torsion-free sheaf $\shf F\in \Coh{\P^2_{I}}$ is $z$-torsion-free.
	\end{lemma}
	\begin{proof}
		For $I= Q_0$, this is \cite[Lemma 3.3.13]{BGK}. For an arbitrary $I$ and a torsion-free $\shf F\in \Coh{\P^2_{I}}$, we simply note that $\tau^T\shf F$ is torsion-free, and so $z$-torsion free by the result just mentioned. Applying the exact functor $\tau_*$ to the monomorphism $(z\cdot)\colon \tau^T\shf F\to \tau^T\shf F(1)$ we obtain a monomorphism $(z_I\cdot)\colon \shf F\to \shf F(1)$.
	\end{proof}

	\begin{lemma}\label{lem:TorFreeSheafRank0IsNil}
		Let $\shf F\in \Coh{\P^2_I}$ be torsion-free, and assume that $c^*\shf F = 0$. Then $\shf F = 0$.
	\end{lemma}
	\begin{proof}
		By \cite[Proposition 3.3.9(10) and Proposition 2.0.9(4)]{BGK}, $\tau^T(\shf F) = 0$. But then $\shf F = \tau_*\tau^T\shf F = 0$.
	\end{proof}

	\subsection{The bijection to quiver varieties}
	In this section we define sets of (isomorphism classes of) framed sheaves on the spaces $\P^2_I$, and show that they carry a canonical bijection to certain Nakajima quiver varieties. 
	
	\begin{definition}\label{def:SetsOfFramedSheaves}
		Choose a finite-dimensional $\Pibullet_{I,0}$-module $V = \bigoplus_{i\in I}V_i$, and a finite-dimensional $\Pibullet_0$-module $W = \bigoplus_{i\in I} W_i$.
		
		Let $\Cohfr{\P^2_{I}}(V, W)$ be the set of isomorphism classes of $W$-framed sheaves $(\shf F, \phi_{\shf F})$ such that $\shf F\in \Coh{\P^2_{I}}$, $\shf F$ is \emph{torsion-free}, and $H^1(\P^2_{I}, \shf F(-1)) = V$.
	\end{definition}
	
	We now need to introduce a condition on $\Pi_0$-modules $V$ depending on the set $I$.
	\begin{definition}\label{def:sufficiency}
		
		We say that $V$ is \emph{sufficient (with respect to $I$)} if, for every $i\not\in I$, we have $2\dim V_i\ge \sum_{i\textrm{ adjacent to }j} \dim V_j$. Here, $i$ \emph{adjacent to} $j$ means that there is a length-1 path in $Q$ between $i$ and $j$.
		
	\end{definition}
	\begin{lemma}\label{lem:DimensionUpperBound}
		Fix a $\Pi_0$-module $W$, and assume that $W_i = 0$ for all $i\not\in I$. 
		Assume that $M$ is a $\theta_I$-stable $\Pi^{\mathbf w}$-module, with $\dim e_\infty M = 1$. Then $\dim M = (1, v)$, where $v\in \Z_{\ge 0}^{Q_0}$ is such that every sufficient $\Pi_0$-module $V$ with $\dim_i V = v_i$ for every $i\in I$ satisfies $\dim V\ge v$.
	\end{lemma}
	\begin{proof}
		This is proved exactly as \cite[Lemma A.2]{CGGS}.
	\end{proof}
	
	\begin{lemma}\label{lem:Sufficiency}
		Assume that $V$ is a finite-dimensional $ \Pi_{I,0}$-module, and assume that $U, U'$ are finite-dimensional $\Pi_0$-modules such that $U|_I \cong  U'|_I\cong V$. Assume furthermore that $\dim U\le \dim U'$, and that $U'$ is sufficient with respect to $I$. Then there is a canonical injective map of sets $f\colon\mathfrak{M}_{\theta_I}(U, W)(\C) \injto \mathfrak{M}_{\theta_I}(U', W)(\C)$. If $U$ is also sufficient with respect to $I$, this map is a bijection.
	\end{lemma}
	\begin{proof}
		Let $\bigoplus_{\alpha\in A} M_\alpha$ be a $\theta_I$-polystable $\Pi^{\mathbf w}$-module of dimension vector $(1, \dim U)$, with every $M_\alpha$ $\theta_I$-stable, representing a point of $\mathfrak{M}_{\theta_I}(U, W)(\C)$. Then there is a unique $\alpha$ (say $\alpha = a$) such that $\dim e_\infty M_a=1$. Because $M_a$ is $\theta_I$-stable, every other $M_\alpha$ is (see \cite[Remark 4.7]{CGGS}) a vertex simple module. Furthermore, by \cref{lem:DimensionUpperBound}, $\dim M_a\le \hat{v}$, where $\hat{v}$ is the dimension of the smallest sufficient $\Pi_0$-module $\hat{V}$ satisfying $\dim e_i\hat{V} = \dim e_i V$ for all $i\in I$. We can then build a $\theta_I$-polystable module of dimension $(1,\dim U')$ as $M_a\bigoplus_{i\not\in I} \oplus S_i^{\dim_i U'-\dim_i {M_a}}$. This provides the desired map. If $U$ is sufficient, this construction can clearly be inverted, and so we have $\mathfrak{M}_{\theta_I}(U, W)(\C) = \mathfrak{M}_{\theta_I}(U', W)(\C)$.
	\end{proof}
	
	\begin{corollary}\label{cor:Sufficiency2}
		Assume that $U $ is a finite-dimensional $\Pi_0$-module. Then, for any $\Pi_0$-module $V$ sufficient with respect to $I$ such that $V|_I = U|_I$, there is a canonical injective map of sets $f\colon\mathfrak{M}_{\theta_I}(U, W)(\C) \injto \mathfrak{M}_{\theta_I}(V, W)(\C)$. 
	\end{corollary}
	\begin{proof}
		Follows immediately from the above.
	\end{proof}
	
	\begin{remark}
		We strongly suspect that \cref{lem:Sufficiency} can be upgraded to an isomorphism of schemes, but we leave this question for future work. (But see \cite[Theorem 1.3]{GammelgaardBertsch} for a cornered version.)
	\end{remark}
	
	The following result provides our path from $\Cohfr{\P^2_I}$ to $\Pi^\mathbf{w}$-modules.
	\begin{proposition}\label{lem:VaVaBij}
		Let $V, W$ be finite-dimensional $\Pi_0$-modules. Then there is a canonical bijection 
		\[\mu_{V,W}\colon \Cohfr{\P^2}(V, W)\isoto \mathfrak M_{\theta}(V,W).\] Furthermore, this bijection is functorial, in the sense that if $V\subset V'$, $\shf F\in \Cohfr{\P^2}(V, W)$ and $\shf F'\in  \Cohfr{\P^2}(V', W)$, with a morphism $f\colon \shf F\to \shf F'$ of framed sheaves, then there is a corresponding morphism of $\Pi^{\mathbf w}$-modules $\phi\colon\mu_{V,W}(\shf F)\to\mu_{V', W}(\shf F')$, such that $\phi\circ\mu_{V,W} = \mu_{V',W}\circ f$.
	\end{proposition}
	
	\begin{proof}
		The first is \cite[Theorem 1]{VaVa}, combined with \cref{lem:EquivCategories}. The functoriality is not shown directly in \cite{VaVa}, but it follows straightforwardly from the construction in \emph{ibid}.
	\end{proof}
	We shall write $\mu $ for $ \mu_{V,W}$ when no confusion is likely.
	
	\begin{theorem}\label{thm:main}
		Choose finite-dimensional $\Pi_{I,0}$-modules $V, W$, and let $V'$ be a sufficient $\Pi_0$-module such that $V'|_I = V$. 
		Consider $W$ as a $\Pi_0$-module where $W_i = 0$ for any $i\notin I$.
		
		There is a canonical bijection $\Cohfr{\P^2_{I}}(V, W)\isoto \mathfrak{M}_{\theta_I}(V', W)(\C)$.
	\end{theorem}
	
	\begin{proof}
		Let $\shf F\in \Cohfr{\P^2_{I}}(V, W)$. Then, by \cref{prop:PreservesFraming}, \cref{prop:realExting}, and \cref{lem:EquivalencesOfTorFreeInvIm}, $\tau^T\shf F\in \Cohfr{\P^2_\Gamma}(U, W)$ for some $U$ such that $ U|_I = V$. Then, by \cref{lem:VaVaBij} $\tau^T\shf F$ corresponds to a closed point $y(\shf F)$ of the Nakajima quiver variety $\mathfrak{M}_{\theta}(U, W)$,  and $y(\shf F)$ maps to a point $\hat g(\shf F)\in \mathfrak{M}_{\theta_I}(U, W)$.

		This defines a map of sets \[\hat{g}\colon \Cohfr{\P^2_{I}}(V, W)\to  \bigsqcup_{U}\mathfrak{M}_{\theta_I}(U, W).\] 
		
		By \cref{cor:Sufficiency2}, there is for every $U$ a canonical map $\mathfrak{M}_{\theta_I}(U, W)\injto \mathfrak{M}_{\theta_I}(V', W)$. Composing these with $\hat{g}$,  
		we obtain a canonical map of sets \[g\colon \Cohfr{\P^2_{I}}(V, W)\to \mathfrak{M}_{\theta_I}(V', W).\]

		We now construct the inverse map to $g$. So let $x\in\mathfrak{M}_{\theta_I}(V', W)$ be a point, and choose a lift of $x$ to a point  $\tilde{x}\in \mathfrak{M}_{\theta}(V', W) $. This point corresponds uniquely (again by \cref{lem:VaVaBij}) to a framed sheaf $\shf G(\tilde x)\in\Cohfr{\P^2_{\Gamma}}(V', W) $, and we have $\tau_*\shf G(\tilde x)\in \Cohfr{\P^2_{I}}(V', W)$. Let us first show the the map $x\mapsto \tau_*\shf G(\tilde x)$ is well-defined: Let $\tilde x_1, \tilde x_2\in \mathfrak{M}_{\theta}(V', W) $ be two different lifts of of $x$, corresponding to framed sheaves $\shf G_1, \shf G_2 $. The $\theta$-stable modules respectively corresponding to $\tilde x_1, \tilde x_2$  are then $\theta_I$-equivalent, and the polystable element of the equivalence class again has a unique direct summand $M_a$ such that $\dim e_\infty M_a=1$. Just as before, $M_a$ must be $\theta$-stable, and then corresponds to a framed sheaf $\shf H\in \Cohfr{\P^2_{\Gamma}}(U, W)$ for some $U$ with $U|_I = V'|_I$. Similarly, we obtain injective morphisms $\shf H\injto \shf G_1, \shf H\injto \shf G_1$. But as above, we find that $\tau_*\shf H = \tau_*\shf G_1 = \tau_*\shf G_2$, therefore the map $f\colon x\mapsto \tau_*\shf G(\hat x)$ is well-defined.

		It is clear that $f(g(\shf F)) = \shf F$ for any $\shf F\in \Cohfr{\P^2_I}(V, W)$, and $g(f(x)) = x$ for any $x\in \mathfrak{M}_{\theta_I}(V', W)$. 
		Thus we have now constructed a bijection $\mathfrak{M}_{\theta_I}(V', W)(\C)\cong \Cohfr{\P^2_I}(V, W)$.
		This concludes the proof.
	\end{proof}
	
	\subsection{Factoring morphisms}
	Consider now the setting of having two sets of representations $I_1\subset I_2\subset Q_0$. Let us write ${j_1}_!, {j_1}^*, {\tau_1}_*, {\tau_1}^*, {\tau_1}^T$ for the functors $j_!, j^*, \tau^*, \tau_*, {\tau}^T$ defined with respect to $I_1$, similarly for the functors defined with respect to $I_2$.
	Let us define 'partial' versions of $\tau_*, \tau^T$: Set ${j_{1,2}}_!\colon \Pibullet_{I_1}\textrm{-}\cat{mod}\to \Pibullet_{I_2}\textrm{-}\cat{mod}$ be the functor given by ${j_{1,2}}_!(M) \coloneqq e_{I_2}\Pibullet e_{I_1}\otimes_{\Pibullet_{I_1}} M$. Similarly, define  ${j_{1,2}}^*\colon \Pibullet_{I_2}\textrm{-}\cat{mod}\to \Pibullet_{I_1}\textrm{-}\cat{mod}$ by ${j_{1,2}}^*(N)\coloneqq e_{I_1}\Pibullet e_{I_2}\otimes_{\Pibullet_{I_2}} N$. 
	\begin{corollary}
		The functors ${j_1}_!, {j_1}^*$ factor as ${j_1}_! = {j_2}_!\circ {j_{1,2}}_!$, respectively ${j_1}^* = {j_{1,2}}^*\circ{j_2}^*$.
	\end{corollary}\begin{proof}
		Obvious from the definition.
	\end{proof}
	
	Because of this factorisation, it is clear that ${j_{1,2}}_!,\ {j_{1,2}}^*$ take finite-dimensional modules to finite-dimensional modules.
	We then have:
	\begin{corollary}\label{cor:factorise}
		The functors ${j_{1,2}}_!,\ {j_{1,2}}^*$ induce functors $\tau_{1,2}^*\colon \Coh{\P^2_{I_1}}\to \Coh{\P^2_{I_2}},\ {\tau_{1,2}}_*\colon \Coh{\P^2_{I_2}}\to \Coh{\P^2_{I_1}}$, and we have factorisations
		\[{\tau_1}_* = {\tau_{1,2}}_*\circ {\tau_{2}}_*,\quad {\tau_1}^* = {\tau_{2}}^*\circ {\tau_{1,2}}^*.\]
		
		Furthermore, define $\tau_{1,2}^T\colon \Coh{\P^2_{I_1}}\to \Coh{\P^2_{I_2}}$ by setting $\tau_{1,2}^T(\shf F)$ to be the quotient of $\tau_{1,2}^*(\shf F)$ by its maximal subsheaf $\shf T$ with $c^*\shf T=0$. Then there are factorisations 
		\[ {\tau_1}^T = {\tau_{2}}^T\circ {\tau_{1,2}}^*,\quad {\tau_1}^T = {\tau_{2}}^T\circ {\tau_{1,2}}^T.\]
	\end{corollary}\begin{proof}
		Also clear, using \cref{prop:PreservesFraming}.
	\end{proof}
	
	\begin{proof}[Proof of \cref{cor:firstFunctorial}]
		Taking into account the result of \cref{cor:factorise} together with those from \cref{sec:CorneredP2Geometry}, we have now proved the statement.
		
	\end{proof}
	
	The point of introducing these functors is to 
	prove the following natural statement:
	\begin{proposition}
		Choose two subsets $I_1\subset I_2\subset Q_0$ and a $\Pi_{I_2,0}$-module $V$. Assume that there is a $\Pi_0$-module $V'$ such that $V'|_{I_2} = V$, and that $V$ is sufficient for for $I_2$.  
		
		Let $\theta_{1,2}$ be the morphism $\mathfrak M_{\theta_2}(V',W)\to \mathfrak M_{\theta_1}(V', W)$ induced by variation of GIT parameter $\theta_{I_2}\rightsquigarrow \theta_{I_1}$.  Let $f_{1,2}$ be the map of sets $\Cohfr{\P^2_{I_2}}(V,W)\to \Cohfr{\P^2_{I_1}}(V|_{I_1},W)$ induced by the functor ${\tau_{1,2}}_*$, and let $\mu_1, \mu_2$ be the bijections given by \cref{thm:main}, for $I_1$, respectively $I_2$. Then the diagram
		\begin{equation}\label{diag:factorisation}
			\begin{tikzcd}
				{\mathfrak M_{\theta_2}(V,W)}(\C) \arrow[d, "{\theta_{1,2}}"] \arrow[r, "\mu_2"] & {\Cohfr{\P^2_{I_2}}(V,W)} \arrow[d, "{f_{1,2}}"] \\
				{\mathfrak M_{\theta_1}(V, W)}(\C) \arrow[r, "\mu_1"]                            & {\Cohfr{\P^2_{I_1}}(V,W)}                       
		\end{tikzcd}\end{equation} commutes.
	\end{proposition}
	
	\begin{proof}
		To do this, simply extend \eqref{diag:factorisation} to the following diagram
		\begin{equation}\label{diag:extendedfactorisation}
			\begin{tikzcd}
				{\mathfrak{M}_{\theta}(V, W)} \arrow[d, "\theta_2"] \arrow[r, "\mu"] \arrow[dd, "\theta_1"', bend right=70] & {\Cohfr{\P^2_{Q_0}}(V,W)} \arrow[d, "f_2"] \arrow[dd, "f_1", bend left=70] \\
				{{\mathfrak M_{\theta_2}(V,W)}(\C)} \arrow[d, "{\theta_{1,2}}"] \arrow[r, "\mu_2"]                          & {\Cohfr{\P^2_{I_2}}(V,W)} \arrow[d, "{f_{1,2}}"]                      \\
				{{\mathfrak M_{\theta_1}(V, W)}(\C)} \arrow[r, "\mu_1"]                                                     & {\Cohfr{\P^2_{I_1}}(V,W)}                                            
		\end{tikzcd}\end{equation}
		where we know, by the proof of \cref{thm:main} that $\mu_2\theta_2 = f_2\mu$ and $\mu_1\theta_1 = f_1\mu$. 
		
		We also know from the proof of \cref{thm:main} that since $V$ is $I_2$-sufficient, the morphism $\theta_2$ is surjective. So assume now that $m\in {{\mathfrak M_{\theta_2}(V,W)}(\C)} $. There is then a $\hat{m}\in \mathfrak{M}_{\theta}(V, W)$ such that $\theta_2(\hat{m}) = m$, and we obtain 
		\[\mu_1\theta_{1,2}(m) = \mu_1\theta_1(\hat{m}) = f_1\mu(\hat{m}) = f_{1,2}f_2\mu(\hat{m}) = f_{1,2}\mu_2\theta_2(\hat{m}) = f_{1,2}\mu_2(m) ,\] as desired.
	\end{proof}

	\section{Connections to previous results}\label{sec:connections}
	In this section we show that \cref{thm:main} generalises previous results on quiver varieties - at least on the level of bijections of closed points.
	
	\subsection{Equivariant Quot schemes}
	Recall that we have defined the two dimension vectors $\delta = n(\dim \rho_i)_{i\in Q_0},\quad \ol 1 = (1, 0,\dots, 0)\in \Z_{\ge 0}^{Q_0}$. Then \cite[Corollary 6.7]{CrawYamagishi} (which strengthens and corrects results of \cite{CGGS2}) shows that $\mathfrak{M}_{\theta_I}(n\delta, \ol 1)$ is isomorphic to an \emph{equivariant Quot scheme} $\QuotInI([\C^2/\Gamma])$.
	This space is (see \cite[Section 3.2]{CGGS2} for the definition) the moduli space of isomorphism classes of $\Pi_I$-module quotients $e_I\Pi e_0\to Q$, such that $\dim e_iQ=n\delta_i$.
	
	We will show that this description agrees, on the level of closed points, with that from \cref{thm:main}.
	
	To start, we define an 'open restriction' functor $r^*\colon \Coh{\P^2_I}\to \Pi_I-\cat{mod}$, constructed very similarly to \cite[Section 5.1]{BGK}. 
	Let $M = \bigoplus_{k\ge 0} M_k$ be a graded $\Pibullet_I$-module, then we can define $\hat{r}^*M = M/(z_I-e_I)M$, which obtains the structure of a $\Pibullet_I/((z_I-e_I)\Pibullet_I) = \Pi_I$-module. 
	
	We then have \[\Pi_I = \hat{r}_*\Pibullet_I = \varinjlim \Pibullet_{I, k},\] where the direct limit is taken with respect to the injective morphisms $(z_I\cdot)\colon \Pibullet_{I,k}\to \Pibullet_{I,k+1}$. This allows us to rewrite $\hat r^*$ as \[\hat r^*\colon M = \bigoplus_{k \ge 0} M_k\mapsto \varinjlim M_k,\] again with the limit taken over the morphisms $(z_I\cdot)\colon M_k\to M_{k+1}$. 
	
	\begin{lemma}\label{lem:OpenResProperties}
		The functor $\hat r^*\colon \Pibullet_I-\cat{mod}\to \Pi_I-\cat{mod}$ factors through $\Coh{\P^2_I}$. Write $r^*\colon\Coh{\P^2_I}\to \Pi_I-\cat{mod} $ for the induced functor.
		Then $r^*$ is exact.
	\end{lemma}
	\begin{proof}
		This is precisely the proof of \cite[Lemma 5.1.2]{BGK}, with $\Pibullet_I, \Pi_I$ in place of  $A^\tau$, respectively $ \mathcal{B}^\tau$ of \emph{loc.cit}.
	\end{proof}
	
	\begin{lemma}\label{lem:TrivAlongFraming}
		Let $\shf Q\in \Coh{\P^2_I}$, with $c^*\shf Q = 0$, and a surjective morphism $\OO_{\P^2_I}e_0\to \shf Q$.
		There is a $\Pi_I$-module $Q'$ such that $\shf Q = \pi(\bigoplus_{k\ge 0} z_I^k Q')$, and $\dim r^*\shf Q = \dim Q'$.
	\end{lemma}
	\begin{proof}
		Choose a graded $\Pibullet$-module $Q = \bigoplus_{k\ge 0} Q_k$ such that $\pi(Q) = \shf Q$ and  $\psi\colon \Pibullet_I e_0\to Q$ is surjective. (This can always be done: if $\psi\colon\Pibullet_I e_0\to Q$ has a finite-dimensional cokernel, replace $Q$ by $\im \psi$. Then $\pi(\im \psi) =  \pi(Q) =\shf Q$.) 
		
		Now $c^*\shf Q = 0$ is equivalent to $\shf Q = z_I\shf Q$, which is equivalent to there being an $m>0$ such that $zQ_k = Q_{k+1}$ for all $k\ge m$. Then $\pi(Q) = \pi(\bigoplus_{k\ge m} Q_{k})$, and so we can assume that $z_I^kQ_0 = Q_k$, and we can take $Q' = Q_0$.
	\end{proof}
	\begin{remark}
		It need not be true that $r^*\shf Q = Q'$, because the actions of the elements of $\Pibullet_I$ not contained in $\{z_i\}_{i\in I}$ on $ \bigoplus_{k\ge m} Q_{k}$ are not necessarily zero.
	\end{remark}
	In  \cite[Corollary 6.7]{CrawYamagishi}, it is shown that there is an isomorphism $\mathfrak M_{\theta_I}(n\delta, \ol 1)\isoto\QuotInI([\C^2/\Gamma])$\todo{is this notation OK?}, where the latter space is the moduli space of quotients $e_I\Pi e_0\to Z\to 0$ of $\Pi_I$-modules $Z$ such that $\dim e_iZ = n_i$. We now show that the bijection on closed points induced by this isomorphism agrees with that shown in \cref{thm:main}.
	\begin{corollary}\label{cor:IsEquivQuotScheme} Let $V$ be a $\Pi_0$-module of dimension $n\delta$, $W$ a $\Pi_0$-module of dimension $\ol 1$.
		Then there is a canonical bijection $\Cohfr{\P^2_I}(V, W)\cong \QuotInI([\C^2/\Gamma])(\C)$.
	\end{corollary}
	\begin{proof}
		
		Let $\shf F\in \Cohfr{\P^2_I}(V, W)$. There is thus an isomorphism $\phi_{\shf F}\colon c^*\shf F\isoto \pi(e_I\Pi e_0)$, and we have $\dim e_iH^1(\P^2_{I}, \shf F(-1)) = n\dim \rho_i$. Let $M = \bigoplus_{k\ge 0}M_k$ be a graded $\Pibullet_I$-module such that $\pi(M) = \shf F$. Note that $\shf F$ is also $z_I$-torsion free by \cref{lem:zTorFree}.
		On the level of modules, the isomorphism $\phi_{\shf F}$ combined with the $z_I$-torsion-freeness 
		means that there is a $k'>0$ such that \begin{enumerate}
			\item $\phi_{\shf F}$ induces an isomorphism $\iota \colon\bigoplus_{k\ge k'} M_{k+1}/z_IM_{k}\isoto \bigoplus_{k\ge k'} (e_I\Pi e_0)_k$
			\item $z_I\colon M_k\to M_{k+1}$ is injective for all $k\ge k'$.
		\end{enumerate}
		Because $\bigoplus_{k=0}^{k'}M_k$ is finite-dimensional, $\pi(M) = \pi(\bigoplus_{k\ge k'}M_k)$. Thus we can replace $M$ by $\bigoplus_{k\ge k'}M_k$.
		
		We then use the injective endomorphism $(z_I\cdot)\colon M\to M$ to construct a \emph{different} grading on $M$, which we will call the $z$-grading. For this grading, set $M_{0_z} \coloneqq M\setminus z_IM$, $M_{k_z} \coloneqq \{0\}\cup \left(z_I^kM\setminus z_I^{k+1}M\right)$. Then $M = \bigoplus_{k\ge 0} M_{k_z}$. It follows that, as $\Pi$-modules, we have \[M_{0_z}= \bigoplus_{k\ge 0} M_{k+1}/z_IM_{k}\overset{\iota}{\isoto} \bigoplus_{k\ge 0} \left((e_I\Pi e_0)(k')\right)_k.\]
		
		We can then construct an injective morphism $\shf F\to \pi( e_I\Pibullet e_0)$ by building a module homomorphism \[t\colon M\to e_I\Pibullet e_0\] as follows: 
		Let $m\in M$. If $m= 0$ set $t(m) = 0$. Otherwise, it is enough to define $t$ for elements homogeneous with respect to the $z$-grading. So take $m\in M_{k_z}$ for some $k$. Thus we can write $m = z_I^{k}m'$ with $m'\in M_{0_z}$.  We then set $t(m) = z_I^k\iota(m')$, viewing $e_I\Pibullet e_0 $ as $ \bigoplus_{k\ge 0} \left(z_I^ke_I\Pi e_0\right) $. It is easily seen that $t$ is a morphism of $\Pibullet_I$-modules. 
		
		Finally, we must show that $t$ is injective, it is again enough to check for $z$-homogeneous elements. So let $m\in M_{k_z}$, and suppose that $t(m) = t(z_I^km') = 0$. Then $0 =t(z_I^k)m' = z_I^kt(m') = z_I^k\iota(m')$, so since $e_I\Pibullet e_0$ is $z$-torsion-free, $\iota(m') = 0$. But $\iota$ is injective, so $m'=0$, and $t$ is injective.
		
		We thus have an injective sheaf morphism $\shf F\injto \OO_{\P^2_{I}} e_0$, and we can make a short exact sequence
		\begin{equation}\label{eq:OpenResSeq} 0\to \shf F \to \OO_{\P^2_{I}} e_0\to \shf Q\to 0.\end{equation} Restricting to $\P^1_{I}$, we find $c^*\shf Q = 0$. Thus $\shf Q = z_I\shf Q$.
		
		Let $Q = r^*\shf Q$, by \cref{lem:TrivAlongFraming} we can assume that $\shf Q = \pi(\bigoplus_{k\ge 0} z_I^k Q')$ with $\dim Q = \dim Q'$.
		
		Applying $r^*$ to \eqref{eq:OpenResSeq}, we find that $Q$ is a quotient of $r^*\OO_{\P^2_{I}}e_0 = \hat{r}^*(e_I\Pibullet e_0) = e_I\Pi e_0$. Now twist \eqref{eq:OpenResSeq} by $(-1)$ and take the following piece of the long exact cohomology sequence:
		\[H^0(\P^2_I, \OO_{\P^2_{I}}e_0(-1))\to H^0(\P^2_I, \shf Q(-1))\to H^1(\P^2_I, \shf F(-1))\to H^1(\P^2_I, \OO_{\P^2_{I}}e_0(-1)).\] Since $H^0(\P^2_I, \OO_{\P^2_{I}}e_0(-1))\subset H^0(\P^2_I,\OO_{\P^2_{I}}(-1)) = 0$ and $H^1(\P^2_I, \OO_{\P^2_{I}}e_0(-1))\subset H^1(\P^2_I,\OO_{\P^2_{I}}(-1)) = 0$, we have 
		\[H^0(\P^2_I, \shf Q(-1)) = H^1(\P^2_I, \shf F(-1)) = V.\] But now we can apply that 
		\[H^0(\P^2_I, \shf Q(-1)) = H^0(\P^2_I, \shf Q) = \Hom_{\Pibullet_I}(\bigoplus z_I^k \Pi_I, \bigoplus z_I^k Q') = \Hom_{\Pi_I}(\Pi_I, Q')\] and so we find that $\dim V_i = \dim e_iV = \dim e_i\Hom_{\Pi_I}(\Pi_I, Q') = \dim e_iQ' = \dim e_iQ$.
		This proves one direction of the required bijection.
		
		For the other, let $Q\in \QuotInI([\C^2/\Gamma])$. Thus $Q$ is (up to isomorphism) a $\Pi_I$-module such that $Q$ is a quotient of $e_I\Pi e_0$, and $\dim e_iQ = n\delta_i$. Let $Q = \pi(\bigoplus_{k\ge 0} z_I^k Q')$, where $\bigoplus_{k\ge 0} z_I^k Q'$ has the induced $\Pibullet_I$-structure. By construction, we have a surjection $\OO_{\P^2_I}e_0\to Q$, let $\shf F$ be the kernel of this morphism. Using the same argument as in the previous paragraph in reverse, we then find that $\dim_i H^1(\P^2_I, \shf F) = \dim e_i Q'$. This concludes the proof.
	\end{proof}
	\subsection{Framed sheaves on a compactification}
	There is (\cite{Paper1}) a projective two-dimensional Deligne-Mumford stack $\mathcal{X}$ with a stacky divisor $d_\infty$, containing $\C^2/\Gamma $ as an open substack, we repeat the definition in \cref{def:Stack}.

	\begin{definition}\label{def:FramedSheavesOnStack}
		A \emph{framed sheaf of rank $r$} on $\mathcal X$ is a pair $(\shf F, \phi_{\shf F})$ of  a torsion-free coherent sheaf of rank $r$ $\shf F$, and an isomorphism $\phi_{\shf F}\colon \shf F|_{d_\infty}\isoto \OO_{d_\infty}^{\oplus r}$.
		
		A morphism $f\colon (\shf F, \phi_{\shf F})\to (\shf G, \phi_{\shf G})$ of framed sheaves of rank $r$ consists of a morphism of sheaves $f\colon \shf F\to \shf G$ such that $\phi_{\shf G}\circ f = \phi_{\shf F}$.
		
		We define a framed sheaf on $\P^2/\Gamma$ by changing $\mathcal X$ to $\P^2/\Gamma$ in the definition above, and changing $d_\infty$ for $r_\infty\coloneqq \{z=0\}/\Gamma$.
	\end{definition}
	We shall often suppress $\phi_{\shf F}$ from notation, and simply write $\shf F$ for $(\shf F, \phi_\shf F)$. 
	
	Now set $\mathbf w = (r, 0,\dots, 0)$.
	The main result of \cite{Paper1} is that $\mathfrak{M}_{\theta_{\{0\}}}(n\delta, \mathbf w)$ parametrises the set $X_{r, n}$ of isomorphism classes of framed torsion-free sheaves $\shf F$ of rank $r$ on $\mathcal{X}$, such that $\dim H^1(\mathcal X, \shf F\otimes \shf I_{d_\infty})\cong n\delta$, where $\shf I_{d_\infty}$ is the ideal sheaf of $d_\infty$.
	
	We will show in Appendix A that this stack is actually unnecessary, and it is enough to use its coarse moduli scheme $\P^2/\Gamma$ in place of $\mathcal X$, and $r_\infty$ in place of ${d_\infty}$. We now show that our description from \cref{thm:main} also generalises this.
	\begin{corollary}\label{cor:IsFramedSheavesOnStack}
		The set $\Cohfr{\P^2_{\{0\}}}(n\delta, \mathbf w)(\C)$ has a canonical bijection to the set $X_{r,n}$.
	\end{corollary}	
	\begin{proof}
		Assume \cref{prop:MainAppA}. We thus wish to show that a $\shf F\in \Cohfr{\P^2_{\{0\}}}(n\delta, \mathbf w)$ corresponds to an isomorphism class of framed sheaves on the (commutative!) scheme $\P^2/\Gamma$. By \cref{prop:CorneredTo0}, we thus have $\P^2_{\{0\}}= \P^2/\Gamma$, and so $\shf F$ can be identified with a coherent sheaf on $\P^2/\Gamma$. The framing isomorphism $\phi_{\shf F}\colon c^*\shf F\isoto \pi((\Pi e_0)^{\oplus r}) $ induces, by \cref{lem:EquivCategories}, an isomorphism of $\shf F|_{r_\infty}$ with  $\OO_{r_\infty}^{\oplus r}$.
		
		As the notions of torsion-free sheaf and sheaf cohomology on $\P^2_{\{0\}}$ and $ \P^2/\Gamma$ agree, the claim follows.
	\end{proof}

	\appendix
	\section{A comparison with a projective stack compactifying $\C^2/\Gamma$}\label{sec:Appendix}
	
	We start by giving an overview of some constructions from \cite{Paper1}.
	Consider the action of $\Gamma$ on $\P^2= \Proj \C[x,y,z]$, and note that it is free on the locus $\P^2\setminus (\{o\}\cup L)$, where $L = \{z=0\}, o= (0:0:1)$. Thus the quotient stack $[(\P^2\setminus (\{o\}\cup L))/\Gamma]$ is represented by the scheme $(\P^2\setminus (\{o\}\cup L))/\Gamma$.
	
	\begin{definition}[{\cite[Definition 3.1]{Paper1}}]\label{def:Stack}
		We set \[\mathcal{X} = [\P^2\setminus\{o\}/\Gamma]\cup_{[(\P^2\setminus(L\cup \{o\}))/\Gamma]} \C^2/\Gamma.\]
	\end{definition}
	
	Then we can show (see \emph{ibid.}) that $\mathcal{X}$ is a projective Deligne-Mumford stack, with a unique singular point, containing $\C^2/\Gamma$ as an open substack, the complement of which is a closed substack $d_\infty $ isomorphic to $[L/\Gamma]$. The coarse moduli space of $\mathcal X$ is $\P^2/\Gamma$, and we set $r_\infty$ to be the scheme-theoretic image of $d_\infty$ in $\P^2/\Gamma$.
	We can build the following commutative diagram of stacks and morphisms:
	\[
	\begin{tikzcd}
		L \arrow[d, "\pi|_L", shift right] \arrow[r, "i"] \arrow[dd, "q|_L"', bend right, shift right] & \P^2 \arrow[d, "\pi"] \arrow[dd, "q", bend left] \\
		{d_\infty} \arrow[d, "k|_{d_\infty]}"] \arrow[r]                                                               & \mathcal X \arrow[d, "k"]                             \\
		r_\infty \arrow[r, "j"]                                                                        & \P^2/\Gamma                                     
	\end{tikzcd}
	\]
	
	There are functors $D\colon \ECoh{\Gamma}{\P^2}\to \Coh{\mathcal X},\quad  \pi^T\colon \Coh{\mathcal X}\to \ECoh{\Gamma}{\P^2}$. We also have that $q_*(-)^\Gamma = k_*\circ D\colon \ECoh{\Gamma}{\P^2}\to \Coh{\P^2/\Gamma}$. Here $\ECoh{\Gamma}{\P^2}$ is the category of $\Gamma$-equivariant coherent sheaves on $\P^2$.
	
	We will need the following result of Mumford.
	\begin{proposition}{{\cite[p. 111]{Mumford}}}\label{thm:descentWhenScheme}
		Let $ \Gamma $ be a finite group acting on a scheme $ X $, such that the orbit of every closed point is contained in an open affine subscheme. 
		Then 
			there is a finite $ \Gamma $-invariant morphism $q\colon X\to X/\Gamma $ which is a categorical quotient, i.e. it has the universal property that any $ \Gamma $-invariant morphism $ X\to Y $ of schemes factors uniquely through $ X/\Gamma $.
			The structure sheaf of $ X/\Gamma $ satisfies \[ \OO_{X/\Gamma}(U)=(\OO_{X}(U))^\Gamma \] for any open $ U\subset X/\Gamma $, i.e. $q_*(-)^\Gamma$ takes $\OO_X$ to $\OO_{X/\Gamma}$.
			
	\end{proposition}
	
	The goal of this Appendix is to prove the following statement.
	\begin{proposition}\label{prop:MainAppA}
		Let $\Cohfr{\mathcal{X}}(V, W)$ be the set of isomorphism classes of $W$-framed sheaves $(\shf F, \phi_{\shf F})$ on $\mathcal{X}$ with $\phi_{\shf F}\colon \shf F|_{d_\infty}\isoto W\otimes \OO_{d_{\infty}}$ and $H^1(\mathcal X, \shf F\otimes \shf I_{d_\infty}) \cong V$. There is then a canonical bijection of sets $\Cohfr{\mathcal{X}}(V, W)\cong \Cohfr{\P^2}(V, W)$.
	\end{proposition}
	\begin{proof}
		Let $\shf F\in \Cohfr{\mathcal{X}}(V, W) $. Because $\P^2/\Gamma$ is the coarse moduli space of $\mathcal X$, it holds that $H^1(\mathcal{X}, \shf F\otimes \shf I_{d_\infty}) = H^1(\P^2/\Gamma,\shf F\otimes \shf I_{r_\infty})$, and that $k_*k^T$ is the identity on $\P^2/\Gamma$. 
		
		So the only thing we need to show is the following: Let $(\shf F, \phi)$ be a framed sheaf on $\mathcal X$ of rank $m$.
		Then $\phi$ induces a framing morphism on $k_*\shf F$, i.e. an isomorphism $k_*\shf F|_{r_\infty}\isoto \OO^{\oplus m}$.
		
		In order to do this, we will first show that $k_*\shf F$ is locally free in some neighbourhood of $r_\infty$.
		Since we only care about neighbourhoods of the framing divisors $d_\infty\subset \mathcal{X}, r_\infty\subset \P^2/\Gamma$, we can without loss of generality restrict to $[(\P^2\setminus \{o\})/\Gamma]\subset \mathcal X$. So then we can take $\shf F$ to be a $\Gamma$-equivariant framed sheaf on $\P^2\setminus \{o\}$, i.e. $\shf F =(\shf F, \phi_{\shf F})$, where $\shf F$ is a coherent $\Gamma$-equivariant torsion-free sheaf on $\P^2$ of rank $m$, and $\phi_{\shf F}$ is an equivariant isomorphism $\shf F|_{L}\isoto \OO_L^{\oplus m}$.
		
		So the diagram we care about is (with a slight abuse of notation)
		
		\begin{equation}\label{diag}
			{ \begin{tikzcd}
					L \arrow[r, "i"] \arrow[d, "q|_L"'] & \P^2\setminus \{o\} \arrow[d, "q"] \\ 	
					L/\Gamma \arrow[r, "j"]                                                                        &(\P^2\setminus \{o\})/\Gamma                                     
			\end{tikzcd} }
			.
		\end{equation}
		
		Let us start with showing that $q_*(\shf F)^\Gamma$ is locally free on some neighbourhood of $L/\Gamma$.
		
		It is shown in \cite[lemma 4.11]{Paper1} that $\shf F$ is locally free in a neighbourhood $U$ of $L$. First, let us replace $U$ by its intersection with all its $\Gamma$-translates, this is simply done to make $U$ $\Gamma$-invariant.
		Let $p$ be a point in $L$. Then we can find an open $V_p$ lying inside $U$, and containing the $\Gamma$-orbit of $p$, such that $V_p$ trivialises $\shf F$.
		(For instance, note that $U$ must be all of $\P^2$ except some finite collection of points. Choose a curve $C$ passing through all these points with $p\not\in C$, and set $V_p$ to be $U\setminus C$, then $V_p$ is affine, and thus must trivialise $\shf F$.) Then we can set $W_p$ to be the intersection of $V_p$ with all its $\Gamma$-translates, especially the $\Gamma$-orbit of $p$ lies in $W_p$. Now $W_p$ trivialises $\shf F$ as a locally free sheaf, but we must show it trivialises $\shf F$ as an \emph{equivariant} locally free sheaf. To do this, note that because $W_p$ is $\Gamma$-invariant and connected, and $\shf F$ is $\Gamma$-equivariant, we must have a $\Gamma$-equivariant isomorphism $ \shf F|_{W_p}\isoto \bigoplus \rho_i\otimes \OO_{W_p}^{m_i} $ for some integers $m_i$ with $\sum (m_i\dim \rho_i) = m$. But restricting to $W_p\cap L$, this isomorphism becomes 
		\[ \phi\colon \OO^m_{W_p\cup L}\isoto \shf F|_{W_p\cap L}\isoto \bigoplus \rho_i\otimes \OO_{W_p\cap L}^{m_i} ,\] thus $m_0=m$, and all the other $m_i$s vanish.
		
		Construct such a $W_p$ for every $p$ in $L$, and set $W \coloneqq \bigcup_{p\in L}W_p$ . (By noetherianness, it is enough to use a finite number of the $W_p$, but we won't need that.) Then $W$ is an open neighbourhood of $L$ on which $\shf F$ is \emph{equivariantly} locally free.
		
		Then $W/\Gamma$ is an open neighbourhood of $r_\infty$, and it is covered by the $W_p/\Gamma$. Now, because we have an isomorphism $t\colon \shf F|_{W_p}\isoto \OO_{W_p}^m$, it follows by \cref{thm:descentWhenScheme} that there is an isomorphism \[t^\Gamma\colon q_*(\shf F)^\Gamma|_{W_p/\Gamma} \isoto \OO_{W_p/\Gamma}^m,\] so $q_*(\shf F)^\Gamma$ is locally free on $W/\Gamma$. We have now shown that $q_*(\shf F)^\Gamma$ is indeed locally free on some neighbourhood of $L/\Gamma$.
		
		We now show the framing, i.e., that $\phi$ induces an isomorphism $j^*(q_*(\shf F))^\Gamma\isoto \OO_{r_\infty}^{m}$.
		Note that we have, in \eqref{diag}, a base extension morphism $(q|_L)_*i^*(\shf F)\to j^*q_*\shf F$, or, taking $\Gamma$-invariants,  $(q|_L)_*(i^*(\shf F))^\Gamma\to (j^*q_*\shf F)^\Gamma = j^*(q_*\shf F)^\Gamma$, and so $\phi$ induces a morphism (again, using \cref{thm:descentWhenScheme} for the action of $\Gamma$ on $L$)
		
		\[  \psi\colon \OO_{r_\infty}^m \overset{\phi^\Gamma}{\isoto} (q|_L)_*(i^*(\shf F))^\Gamma\to j^*(q_*\shf F)^\Gamma.\]
		
		We will show that $\psi$ locally is an isomorphism, and thus an isomorphism. Choose an open equivariantly trivialising subscheme $W_p$ as before, we can then restrict the diagram \eqref{diag} to give (here we are again slightly abusing notation)
		\begin{equation}\label{diagRes}
			{ \begin{tikzcd}
					L\cap W_p \arrow[r, "i"] \arrow[d, "q|_L"'] & W_p \arrow[d, "q"] \\ 	
					(L\cap W_p)/\Gamma \arrow[r, "j"]                                                                        &W_p/\Gamma                                     
			\end{tikzcd} }
			.
		\end{equation}
		
		But because there is a $\Gamma$-equivariant isomorphism $\shf F|_{W_p}\isoto \OO_{W_p}^m$, it is easy to see (again using \cref{thm:descentWhenScheme}) that we have $(q_*\shf F|_{W_p})^\Gamma) \isoto \OO_{W_p/\Gamma}^m$, and so we have 
		\[ \psi|_{W_p\cap r_\infty}\colon \OO_{L\cap W_p/\Gamma}^m \overset{\phi^\Gamma}{\isoto} (q|_L)_*(i^*(\shf F))^\Gamma\to j^*(q_*\shf F)^\Gamma) \isoto \OO_{L\cap W_p/\Gamma}^m .\]
		This concludes the proof.
	\end{proof}
	
	\printbibliography
	
\end{document}